\documentclass[10pt]{amsart}

\usepackage[T1]{fontenc}
\usepackage{amsmath,amsthm,color,mathrsfs,amsfonts,amssymb,amscd,amsgen,latexsym,euscript,overpic}
\usepackage[normalem]{ulem}
\usepackage{xcolor}
\definecolor{dgreen}{rgb}{0.1,0.6,0.}

\def\bZ{\mathbb{Z}}
\def\bR{\mathbb{R}}

\def\cM{\mathcal{M}}
\def\cR{\mathcal{R}}

\def\cF{\mathscr F}
\def\cG{\mathscr G}
\def\cS{\mathcal S}
\def\cI{\mathcal I}
\def\cA{\mathcal A}
\def\cW{\mathscr{W}}

\def\epsilon{\varepsilon}

\def\c{{\rm c}}
\def\s{{\rm s}}
\def\ss{{\rm ss}}
\def\uu{{\rm uu}}
\def\u{{\rm u}}
\def\cs{{\rm cs}}
\def\cu{{\rm cu}}
\def\Sak{{\rm S}}

\newtheorem{thmy}{Theorem}

\newtheorem{theo}{Theorem}[section]
\newtheorem{theorem}[theo]{Theorem}
\newtheorem{corollary}[theo]{Corollary}
\newtheorem{proposition}[theo]{Proposition}
\newtheorem{lemma}[theo]{Lemma}
\newtheorem*{claim}{Claim}
\theoremstyle{definition}

\newtheorem{remark}[theo]{Remark}
\newtheorem*{remark*}{Remark}
\newtheorem{definition}[theo]{Definition}
  
\newcommand{\eqdef}{\stackrel{\scriptscriptstyle\rm def}{=}}

\DeclareMathOperator{\lAngle}{\langle\hspace{-0.2cm}\langle}
\DeclareMathOperator{\rAngle}{\rangle\hspace{-0.2cm}\rangle}

\DeclareMathOperator{\card}{{\rm card}}

\DeclareMathOperator{\Span}{span}

\title[Geodesic flows modeled by expansive flows]{Geodesic flows modeled by expansive flows: Compact surfaces without conjugate points and continuous Green bundles}


\author[K. Gelfert]{Katrin Gelfert}
\address{Instituto de Matem\'atica, Universidade Federal do Rio de Janeiro, Cidade Universit\'aria - Ilha do Fund\~ao, Rio de Janeiro 21945-909,  Brazil}
\email{gelfert@im.ufrj.br}

\author[R.~O.~Ruggiero]{Rafael O.~Ruggiero}
\address{Departamento de Matem\'atica PUC-Rio, Rua Marqu\'es de S\~ao Vicente 225, Rio de Janeiro 22543-900, Brazil}
\email{rorr@mat.puc-rio.br}

\thanks{This research has been supported [in part] by the Coordenação de Aperfeiçoamento de Pessoal de Nível Superior - Brasil (CAPES) - Finance Code 001 and by CNPq-grants. }

\begin{document}

\begin{abstract}
We study the geodesic flow of a compact surface without conjugate points and genus greater than one and continuous Green bundles. Identifying each strip of bi-asymptotic geodesics  induces an equivalence relation on the unit tangent bundle. Its quotient space  is shown to carry the structure of a 3-dimensional compact manifold. This manifold carries a canonically defined continuous  flow which is expansive, time-preserving semi-conjugate to the geodesic flow, and has a local product structure. An essential step towards the proof of these properties is to study  regularity properties of the horospherical foliations and to show that they are indeed tangent to the Green subbundles. As an application it is shown that the geodesic flow has a unique measure of maximal entropy. 
\end{abstract}
\maketitle

\section{Introduction}
The geodesic flow of a compact surface without conjugate points whose genus is greater than one belongs to the most challenging examples of nonuniformly hyperbolic dynamics. From the point of view of topological dynamics, any such flow can be considered ``hyperbolic in the large'' after Morse's work \cite{kn:Morse} which shows that geodesics in the universal covering space, endowed by the pullback of the metric of the surface by the covering map, are ``shadowed'' by \emph{hyperbolic geodesics}, that is, geodesics of the hyperbolic space. To be more precise, a rectifiable curve $c\colon I \to N$, $I$ an interval, of a complete Riemannian manifold $(N,g)$ is a \emph{$A,B$-quasi-geodesic} if for every $t , s \in I$ it holds  $\ell_{g}(c[s,t]) \leq A d_{g}(c(s),c(t)) + B$, where $\ell_g$ denotes curve length and $d_{g}$ the distance relative to the Riemannian metric $g$. Morse shows that if $(N,g)$ is the hyperbolic plane, then there exists $D>0$ such that the curve $c$ is within a distance $D$ from a hyperbolic geodesic.
The term ``shadowing'' is used to somehow draw a connection to the Anosov-shadowing lemma  in hyperbolic dynamics  (see, for instance, \cite[Section 18]{KatHas:95}) which asserts that any $\epsilon$-pseudo-orbit (for $\epsilon$ small enough) is shadowed by some true orbit of the dynamics. In some sense quasi-geodesics play a role analogous to pseudo-orbits of Anosov dynamics and the constant $D$ replaces $\epsilon$ in the Anosov-shadowing lemma. 

Even though, by the above, geodesics behave similar to hyperbolic geodesics, there is a fundamental difference: there might exist infinitely many geodesics in the universal covering of the compact surface shadowed by just a single hyperbolic geodesic. These geodesics form ``strips'' of bi-asymptotic geodesics which have been the object of study of dynamicists working in geometry. One of the most famous results is the so-called ``flat strip theorem'' for surfaces without focal points (see  \cite{kn:EO,kn:Pesin} and discussion in Section \ref{expansivepoints}). 

The similarities between the dynamics of the geodesic flow of a surface without conjugate points and genus greater than one and the geodesic flow of a hyperbolic surface have been inspiration in the fields of dynamical systems theory, geometry, and topology. Among the most studied problems is the existence of conjugacies or semi-conjugacies between these flows, a problem which arises naturally from Morse's work. It was shown in \cite{kn:Ghys,Gro:00} that there exist such semi-conjugacies provided one allows for a reparametrization of the geodesic flow. On the other hand, after the works \cite{kn:Otal,kn:Croke,CroFatFel:92}  on  rigidity of the marked length spectrum it is known that such semi-conjugacies in general cannot be time-preserving. 
It is natural, although somewhat  naive, to ask whether there exists a sort of equivalence relation in the class of orbits of the geodesic flow assigning any strip of geodesics one common equivalence class such that the induced quotient space of the unit tangent bundle  still has some nice metric properties and carries a continuous flow. Without any further hypotheses, presumably the structure of strips is quite complicate. On partial result in this direction is Coud\`{e}ne-Schapira \cite{CouSch:14} stating that in the universal covering of a compact surface without focal points and genus greater than one the only nontrivial strips project under the covering map on cylinders which are completely foliated by closed geodesics. Even though, \emph{a priori} there may be infinitely many strips to ``quotient'' and the quotient space may be quite singular. A general structure may be described by the equivalence relation in Gromov hyperbolic spaces investigated by Gromov \cite[Section 8.3]{kn:Gromov} obtaining a quotient with some very  mild topological structure only. 

Towards this direction, the following is our first main result. We recall the definitions of the corresponding topological concepts in Section \ref{sec:quotient}. 

\begin{thmy} \label{main}
	Let $(M,g)$ be a $C^\infty$ compact connected boundaryless Riemannian surface  without conjugate points of genus greater than one and with continuous stable and unstable Green bundles. 
	Then there exists  a continuous flow of a compact topological 3-manifold which is expansive, topologically mixing, has a local product structure, and is time-preserving semi-conjugate to the geodesic flow of $(M,g)$.
\end{thmy}

Theorem \ref{main} generalizes \cite{GelRug:19} which put the more restrictive assumption that $(M,g)$ is a compact surface without focal points and genus greater than one. 
Let us in the following discuss our hypotheses and some of the main ingredients for its proof.

Green bundles (bundles of stable (resp. unstable) Green Jacobi fields, see definition in Section \ref{sec:Greenss}) are one of the main tools when studying smooth aspects of the dynamics of geodesic flows. Their existence is a special feature of manifolds without conjugate points and more generally of globally minimizing objects of Lagrangian dynamics (Aubry-Mather theory). One immediate consequence of their definition is that Green bundles  are measurable and invariant under the action of the differential of the geodesic flow. By Eberlein \cite{kn:Eberlein}, their linear independence is equivalent to the property that the geodesic flow is an Anosov flow. The hypothesis in Theorem \ref{main} about continuity of Green bundles  is an additional restriction in the setting of manifolds without conjugate points, and it does not grant \emph{a priori} their linear independence. In examples such as manifolds without focal points (and hence in nonpositive or negative curvature) Green bundles are continuous and in fact have an ``expected'' asymptotic behavior (the stable Green bundle is a counterpart of center stable dynamics, the unstable Green bundle of the center unstable one). Anosov \cite{Ano:69}
 shows that Green bundles coincide with the dynamical invariant bundles of hyperbolic dynamics if the compact manifold has negative curvature. 
Very much as in the classification into \emph{regular} or \emph{rank one} vectors and \emph{singular} or \emph{higher rank} vectors in compact manifolds of nonpositive curvature, here we consider two distinguished sets of vectors of the unit tangent bundle:
\[
	\cR_1
	\eqdef \{\theta\in T^1M \text{ has linear independent Green bundles}\}
\] 
the sets of \emph{generalized rank one vectors} and 
\[
	 \cR_0
	\eqdef\{\theta\in T^1M \text{ defines a trivial strip}\}
	\supset \cR_1,
\]
the set of \emph{expansive vectors}. Note that both are invariant under the geodesic flow. 

A crucial issue in the theory of manifolds without conjugate points is the regularity of the horospheres in the universal covering of the manifold. It is not known whether  horospheres give rise to continuous foliations of the unit tangent bundle of the universal covering, invariant by the geodesic flow, as it is the case in Anosov dynamics. The case of compact \emph{surfaces} without conjugate points is quite special since geodesic rays diverge in the universal covering \cite{kn:Green} and since this property is equivalent to the existence and continuity of horospherical foliations \cite{kn:RuggieroA}. In the more special case of compact \emph{nonpositively curved} surfaces this was shown by Eberlein (see \cite{HeiimH:77}), moreover, in this case  Green bundles vary continuously \emph{and} are tangent to the horospherical foliations. However, in a more general setting (even for compact surfaces without conjugate points) it is not known if the latter remains true. What is known is that a ``tame asymptotic behavior'' of Green bundles usually implies that horospherical foliations exist and are tangent to Green bundles (see, for example, \cite[Part II]{kn:Pesin} and discussion in Section \ref{sec:Greenss}). 

As part of the proof of Theorem \ref{main}, but interesting in itself, the following result states that the continuity of Green bundles implies that the horospherical foliations $\cF^\s$ and $\cF^\u$ (see Section \ref{sec:21} for definition) are continuous foliations by $C^{1}$ leaves, tangent to the Green bundles. It hence justifies the terms \emph{stable} and \emph{unstable foliations} for $\cF^\s$ and $\cF^\u$, respectively. 

\begin{thmy}\label{thmy:manifolds}
	Under the hypotheses of Theorem \ref{main}, it holds:
\begin{itemize}
\item[(1)] The families $\cF^\s$ and $\cF^\u$ are continuous foliations by $C^1$ curves which are tangent to the stable and the unstable Green bundles, respectively.
\item[(2)] The set $\cR_1$ coincides with the set of vectors $\theta\in T^1M$ such that $\cF^\s(\theta)$ and $\cF^\u(\theta)$ intersect transversally at $\theta$.
\item[(3)] The set $\cR_1$ is invariant, open, and dense in $T^1M$. 	
\item[(4)] Any vector $\theta\in T^1M$ with positive (forward or backward) Lyapunov exponent belongs to $\cR_1$.
\end{itemize}
\end{thmy}

Theorem \ref{thmy:manifolds} (3) extends previous results for compact surfaces with no focal points \cite{kn:Pesin} and with bounded asymptote \cite{kn:RosasRug}. It is not known if the continuous Green bundles-hypothesis alone implies any controlled asymptotic behavior of Green Jacobi fields as it does in those cases. 
Note that in general (for example assuming that the surface has nonpositive curvature and does not have an Anosov geodesic flow) there exist vectors in $\cR_1$ with Lyapunov exponent zero (see, for example, \cite{Gel:19}). 

 Theorem \ref{thmy:manifolds} will play a crucial role in the proof of the existence of a 3-dimensional manifold carrying an expansive flow time-preserving, semi-conjugate to the geodesic flow. Although the internal structure of  strips  (classes of bi-asymptotic geodesics) may be quite complicated, nevertheless we obtain --up to time-preserving semi-conjugacy-- a model which describes well the dynamics of the geodesic flow under consideration. 

Returning to the term \emph{nonuniformly hyperbolic dynamics} coined in the beginning, we remark that the relevance of Green bundles was settled after the work by Freire-Ma\~{n}\'{e} \cite{kn:FM}. It draws a connection between the Lyapunov spectrum of the geodesic flow, Green bundles, and the calculation of the metric entropy of the Liouville measure (see also Section \ref{sec:greenspas}). Indeed, negative Lyapunov exponents are associated to stable Green bundles while positive exponents are associated to unstable ones. It is unknown if the converse is true. Even under the assumption of their continuity, Green bundles have no \emph{a priori} prescribed asymptotic behavior and its analysis still remains one of the most subtle issues and challenges of the theory of manifolds without conjugate points.

Under the hypotheses of Theorem \ref{main}, a straightforward combination of the variational principle for entropy \eqref{eq:VP} and Ruelle's inequality \eqref{eq:Ruelle} for the positive topological entropy geodesic flow yields the existence of hyperbolic ergodic measures with large metric entropy. 
As an application, also using Bowen's work about thermodynamical formalism, we show the following result.

\begin{thmy} \label{main2}
	Under the hypotheses of Theorem \ref{main}, the entropy map for the geodesic flow is upper semi-continuous and there is a unique measure of maximal entropy.
\end{thmy}

Existence and uniqueness of measures of maximal entropy for nonuniformly hyperbolic dynamical systems have been subject of interest in ergodic theory and dynamical systems theory since the 1960s. Knieper \cite{kn:Knieper} brought attention to the subject in the context of geodesic flows proving that for compact rank one manifolds of nonpositive curvature the geodesic flow has a unique measure of maximal entropy. His proof is based on the construction and study of a Patterson-Sullivan measure and was extended in \cite{LiuWanWu:} to compact rank one manifolds without focal points and in \cite{Bos:15} to compact manifolds  without conjugate points and expansive geodesic flow. 
Recently, Climenhaga et al.  \cite{CliKniWar:} generalizes Knieper's work to compact surfaces without conjugate points. There, they essentially follow an extension of Bowen's classical construction \cite{Bow:74} of maximizing measures for expansive homeomorphisms (see \cite{Fra:77} in the case of expansive continuous flows). 
Theorem \ref{main2} for compact surfaces without focal points was shown in \cite{GelRug:19}, and here we largely will follow the strategy developed therein. 
Our approach, in some essential points different from \cite{CliKniWar:}, relies on a direct application of Bowen-Franco's method for expansive dynamics. Once we have Theorem \ref{main}, the expansive model for the geodesic flow of the surface satisfies the assumptions required to conclude that the expansive model has a unique measure of maximal entropy. Then we apply criteria for extensions of expansive dynamics  in \cite{BuzFisSamVas:12} to carry over the uniqueness of the measure of maximal entropy to the extending flow, proving Theorem \ref{main2}.

The paper is organized as follows. In Section \ref{sec:2} we recall some geometric preliminaries, in particular, in Section \ref{sec:Greenss} we define Green bundles.   In Section \ref{expansivepoints} we properly define the above somewhat vaguely introduced term \emph{strip} and investigate properties of the set of generalized rank one vectors. We also study the set of generalized rank one vectors in this section and prove Theorem \ref{thmy:manifolds}, except for item (4) whose proof we postpone to Section \ref{sec:LyapRicc}. In Section \ref{sec:quotient} defines an equivalence relation between vectors of the unit tangent bundle which correspondingly defines a quotient space and quotient flow. The proof of Theorem \ref{main} will be consequence of Theorems \ref{teo:quotient} and \ref{teo:quotientflow} which are proved in Section  \ref{sec:quotient}. 
Section \ref{sec:LyapRicc} discusses the relation between Lyapunov exponents and Green bundles. In Section \ref{sec:entropy} we study the entropy of the geodesic flow on the set of generalized rank one vectors and prove Theorem \ref{main2}. 

\section{Preliminaries}\label{sec:2}

\noindent\emph{Standing Assumption.}
Throughout the paper $(M,g)$ will be a compact connected $C^{\infty}$ Riemannian manifold without boundary and dimension $n$.
We shall always assume that $M$ has \emph{no conjugate points}, that is, the exponential map is nonsingular at every point.
\smallskip

Our main result concerns surfaces, though many statements hold in any dimension.

Each vector $\theta\in TM$ in the tangent bundle of $M$ determines a unique geodesic $\gamma_\theta(\cdot)$ such that $\dot\gamma_\theta(0)=\theta$. The geodesic flow $\Phi=(\phi_t)_{t\in\bR}$ acts by $\phi_t(\theta)=\dot\gamma_\theta(t)$. We shall study its restriction to the unit tangent bundle $T^1M$, which is an invariant subset of $TM$. All the geodesics will be parametrized by arc length.

We shall denote by $\tilde{M}$ the universal covering of $M$ and endow it with the pullback $\tilde g$ of the metric $g$ by the covering map $\pi\colon\tilde{M} \to M$ which gives the Riemannian manifold $(\tilde{M}, \tilde{g})$.
We shall also consider the geodesic flow of this manifold which acts on $T_1\tilde M$  which we will also denote by $\Phi=(\phi_t)_{t\in\bR}$ (the domain of the flow is enough to specify the dynamical system under consideration).
We will consider the natural projection $\bar{\pi} \colon T_{1}\tilde{M} \to T^1M$.
The distance associated to the Riemannian metric $g$ will be denoted by $d_{g}$ and the one associated to $\tilde{g}$ by $d_{\tilde{g}}$. We will omit the metric if there is no danger of confusion.

Given $\theta=(p,v)$, we recall the natural isomorphism between the tangent space $T_vTM$ and $T_pM\oplus T_pM$ via the isomorphism $\xi\mapsto (D\mu(\xi),C(\xi))$, where $\mu\colon TM\to M$ is the canonical projection $\mu(p,v)=p$ and $C\colon TTM\to TM$ is the connection map defined by the Levi-Civita connection. 
One refers to the orthogonal decomposition of $T_\theta TM$ into the \emph{horizontal} and the \emph{vertical} subspace $T_\theta TM=H_\theta\oplus V_\theta$, respectively.  The Riemannian metric on $M$ lifts to the \emph{Sasaki metric} on $TM$ induced by the scalar product structure which we denote by $d_\Sak$ and which is induced by the following scalar product: for $\xi,\eta\in T_vTM$
\[
	\lAngle \xi,\eta\rAngle_v
	=\langle D\mu_v(\xi),D\mu_v(\eta)\rangle_p+\langle C_v(\xi),C_v(\eta)\rangle_p.
\]

\subsection{Jacobi fields}

The notion of conjugate points has variational origin.  Recall that the Jacobi equation of a geodesic $\gamma_{\theta}$ of $(M,g)$ is given by
\begin{equation}\label{eq:Jacobi}
	J''(t) + R(J(t),\dot\gamma_{\theta}(t)) \dot\gamma_{\theta}(t) =0 ,
\end{equation}
where $R$ denotes the curvature tensor and ${}'$ 
denotes covariant differentiation along $\gamma_{\theta}$. Solutions of
equation \eqref{eq:Jacobi} are called \emph{Jacobi fields}. The Jacobi equation arises in the study of the second variation of the length function of smooth curves \cite{kn:DoCarmo}. If $J$ is a Jacobi field along a geodesic $\gamma$ so that $J(t)$ and $J'(t)$ are orthogonal to $\dot\gamma(t)$ for some $t$ (and hence for all $t\in \bR$) then it is called \emph{orthogonal}.

Let $\gamma_{\theta}$ be a geodesic of $(M,g)$. Two points $\gamma_{\theta}(t)$, $\gamma_{\theta}(s)$, $r\neq s$, are \emph{conjugate} along $\gamma_{\theta}$ if there exists a nontrivial Jacobi field $J(r)$ of $\gamma$ which vanishes at $r=t$ and at $r=s$. The geodesic $\gamma_{\theta}\colon (a,b) \to M$ \emph{has no conjugate points} if every nontrivial Jacobi field of $\gamma_{\theta}$ has at most one zero in $(a,b)$. The manifold $(M,g)$ has \emph{no conjugate points} if and only if no geodesic has conjugate points.

Given $\theta=(p,v)$ and $\xi\in T_vTM$, the Jacobi field $J_\xi$ along $\gamma_\theta$ is uniquely determined by its initial conditions $(J_\xi(0),J_\xi'(0)) = (d\mu(\xi), C_v(\xi))\in T_pM\oplus T_pM$. The above described isomorphism acts as $D\phi_t(\xi)\mapsto(J_\xi(t),J_\xi'(t))$ and, in particular,
\[
	\lVert D\phi_t(\xi)\rVert_v^2
	= \lVert J_\xi(t)\rVert_p^2+\lVert J_\xi'(t)\rVert_p^2.
\]
As $(M,g)$ is compact, the curvature is bounded from below by $-\kappa^2\le K$ for some $\kappa>0$. By \cite[Proposition 2.11]{kn:Eberlein1}, it holds $\lVert J_\xi'(t)\rVert\le\kappa\lVert J_\xi(t)\rVert$ and hence
\begin{equation}\label{eq:Lya}
	\lVert J_\xi(t)\rVert
	\le \frac{\lVert D\phi_t(\xi)\rVert_v}{\lVert\xi\rVert_v}
	\le \sqrt{1+\kappa^2}\lVert J_\xi(t)\rVert.
\end{equation}
By the above, there is an intimate relation between Lyapunov exponents and the growth of  nonradial Jacobi fields. We will use this in Section \ref{sec:greenspas}.

In Section \ref{sec:Greenss} we will introduce a distinguished family of Jacobi fields which define the stable and unstable Green bundles. To do so, we need first to discuss some further ingredients.

\subsection{Horospheres and un-/stable submanifolds}\label{sec:21}

A very special property of manifolds with no conjugate points is the
existence of the Busemann functions and horospheres (see, for example,  \cite[Part II]{kn:Pesin} or \cite{kn:Eschenburg} for details).
Given \( \bar\theta = (p,v) \in T_{1}\tilde{M} \), the (\emph{forward} and \emph{backward}) \emph{Busemann functions} \( b_{\bar\theta}^\pm\colon \tilde{M} \to \mathbb R \)
associated to  \( \bar\theta\) are defined by
\[
	b_{\bar\theta}^+(x) \eqdef \lim_{t\to  +\infty}
		d_{\tilde g}(x,\gamma_{\bar\theta}(t)) - t \\
	\quad\text{ and }\quad
	b_{\bar\theta}^-(x) \eqdef \lim_{t\to  +\infty}
	d_{\tilde g}(x,\gamma_{\bar\theta}(-t)) - t \,,
\]
respectively.
For every $\bar\theta$, the Busemann functions $b_{\bar\theta}^\pm$ are $C^1$ functions with $L$-Lipschitz continuous derivative (with $L>0$ being a Lipschitz constant depending on curvature bounds, see \cite[Propositions 1 and 2]{kn:Eschenburg} and also \cite[Satz 3.5]{kn:Knieper}), the gradients $\nabla b_{\bar\theta}^\pm$ are Lipschitz continuous unit vector fields. 
The level sets of the Busemann functions are the horospheres.
We define the (level $0$) (\emph{positive} and \emph{negative}) \emph{horospheres} of $\bar\theta \in  T_{1}\tilde{M}$ by
\[
	H^+(\bar\theta)\eqdef (b_{\bar\theta}^+)^{-1}(0)
	\quad\text{ and }\quad
	H^-(\bar\theta)\eqdef (b_{\bar\theta}^-)^{-1}(0)
	,
\]	
respectively.
Every horosphere is an embedded submanifold of $\tilde M$ of dimension $n-1$ tangent to a Lipschitz plane field.

Let us denote by \( \sigma^{\bar\theta}_t \colon \tilde{M} \to \tilde{M} \) the integral flow of the vector field \( -\nabla b_{\bar\theta}^+ \) (also called \emph{Busemann flow}). The orbits of this flow are the {\em Busemann asymptotes} of $\gamma_{\bar\theta}$. They are geodesics which are everywhere orthogonal to the horosphere $H^+(\bar\theta)$.
In particular, the geodesic \( \gamma_{\bar\theta} \) is an orbit of this
flow and for every $t\in \mathbb R$ we have
\[
	\sigma^{\bar\theta}_t( H^+(\bar\theta))= H^+(\gamma_{\bar\theta}(t)) .
\]

Geodesics $\beta$ and $\gamma$ in  $\tilde{M}$ are \emph{asymptotic} (as $t\to\infty$) if $d_{\tilde g}(\beta(t),\gamma(t))$ is bounded for $t\ge0$, that is, there exists $C>0$ such that $ d_{\tilde g}(\beta(t),\gamma(t))\le C$ for all $t\ge0$, and \emph{bi-asymptotic} if $d_{\tilde g}(\beta(t),\gamma(t))$ is bounded as $t\to\pm\infty$, that is, the previous inequality holds for all $t\in\bR$. Being asymptotic is an equivalence relation and we denote by $\partial\tilde M$ the set of equivalence classes (the \emph{points at infinity}). Given a geodesic $\beta$, we denote by $\beta(\infty)$ its equivalence class and by $\beta(-\infty)$ the equivalence class of the geodesic $\gamma(t)=\beta(-t)$. By \cite{Kli:71}, for every pair of distinct points in $\partial\tilde M$ there exists a (not necessarily unique) geodesic $\beta$ such that $\beta(\infty)$ and $\beta(-\infty)$ are those points at infinity, respectively. 

If $\beta \subset \tilde{M}$ is a geodesic  such that $\beta$ and $\gamma_{\bar\theta}$ are asymptotic, then $\beta$ is (up to reparametrization) a Busemann asymptote of $\gamma_{\bar{\theta}}$. Moreover, if $\inf_{t >0} d_{\tilde g}(\gamma_{\bar{\theta}}(t), \beta(t) ) =0$, then $\beta$ is a Busemann asymptote and $\beta(0) \in H^+(\bar\theta)$. 

Horospheres are equidistant in the sense that, given any point $p \in H^+(\gamma_{\bar\theta}(t))$, the distance $d_{\tilde g}(p, H^+(\gamma_{\bar\theta}(s)))$ is equal to $\vert t-s \vert$. In particular, \( H^+(\gamma_{\bar\theta}(t))\) varies continuously with $t \in \mathbb R$, however it is not known whether horospheres depend continuously (in the compact-open topology%
\footnote{The map $\bar\theta\mapsto H^\pm(\bar\theta)$ is \emph{continuous} (in the compact-open topology) if given a compact ball $B(q,r) \subset \tilde M$ centred at $q$ and of radius $r$ and $\varepsilon>0$, there exists $\delta=\delta(r,q,\varepsilon)$ such that $\lVert\bar\theta-\bar\iota\rVert\le\delta$ implies $d_{\tilde g}\big(H^\pm(\bar\theta)\cap B(q,r),H^\pm(\bar\iota)\cap B(q,r)\big)\le\varepsilon$.}) on their defining vector.
The continuity of $\bar\theta\mapsto H^\pm(\bar\theta)$ is equivalent to the continuity in the $C^1$ topology of the map $\bar\theta\mapsto b^\pm_{\bar\theta}$ uniformly on compact subsets of $\tilde M$. 
By \cite{kn:RuggieroA}, for $(M,g)$ a compact manifold without conjugate points, the latter continuity is equivalent to \emph{uniform divergence of geodesic rays} in $(\tilde M,\tilde g)$.%
\footnote{\emph{Geodesic rays diverge uniformly} if for every $\varepsilon>0$, $L>0$ there exist $s=s(\varepsilon,L)>0$ such that for every pair of vectors $(p,v),(p,w)\in T_1\tilde M$ such that $\angle (v,w)\ge\varepsilon$ for every $t\ge s$ we have $d_{\tilde g}\big(\gamma_{p,v},\gamma_{p,w}\big)\ge L$.}

The case of compact \emph{surfaces} is special.
The divergence of geodesic rays in the universal covering of a compact surface without conjugate points was shown by Green \cite{kn:Green}. In higher dimensions the divergence of geodesic rays in the universal covering of compact manifolds without conjugate points still remains an open question.

The horospheres in $\tilde M$ lift naturally to $T_1\tilde M$ as follows.
Consider the gradient vector fields $\nabla b_{\bar\theta}^\pm$ and define the \emph{positive horocycle} $\tilde{\mathscr F}^{\s}(\bar\theta)$ and the \emph{negative horocycle} $\tilde{\mathscr F}^\u(\bar\theta)$  in $T_1\tilde M$ through $\bar\theta$ to be the restriction of $\nabla b_{\bar\theta}^\pm$ to $H^\pm(\bar\theta)$
\[\begin{split}
	\tilde{\mathscr F}^\s (\bar\theta)
	&\eqdef \big\{(q,-\nabla_{q}b_{\bar\theta}^+)\colon q \in H^+(\bar\theta)\big\}
	\quad\text{ and }
\\	\tilde{\mathscr F}^\u (\bar\theta)
	&\eqdef \big\{(q,\nabla_{q}b_{\bar\theta}^-)\colon q \in H^-(\bar\theta)\big\},
\end{split}\]
respectively.

\begin{remark}\label{rem:smoothmfds}
	As recalled above, Busemann functions are $C^1$ with Lipschitz continuous derivative  (with Lipschitz constant depending on curvature bounds). Each $\tilde{\mathscr F}^{\s}(\bar\theta)$ (each $\tilde{\mathscr F}^{\u}(\bar\theta)$) is the union of the vectors of the unit vector field being normal to the horosphere $H^+(\bar\theta)$ (to $H^-(\bar\theta)$), and hence a continuous $(n-1)$-dimensional submanifold of $T_1\tilde M$.
\end{remark}

By definition, the families $\{\tilde{\mathscr F}^\s(\bar\theta)\}_{\bar\theta\in\tilde M}$ and $\{\tilde{\mathscr F}^\s(\bar\theta)\}_{\bar\theta\in\tilde M}$ both are invariant in the sense that for every $\bar\theta$ and every $t\in\bR$ it holds 
\[
	\phi_t(\tilde{\mathscr F}^\s (\bar\theta))
	= \tilde{\mathscr F}^\s (\phi_t(\bar\theta))\\
	\quad\text{ and }\quad\phi_t(\tilde{\mathscr F}^\u (\bar\theta) )
	= \tilde{\mathscr F}^\u (\phi_t(\bar\theta)).
\]

When $M$ has nonpositive curvature this family provides a continuous foliation of $T_1\tilde M$ \cite{kn:Eschenburg}. In the particular case of a  compact surface without conjugate points each leaf of this foliation is a Lipschitz leaf (this is a consequence of the divergence of geodesic rays in the universal cover due to Green \cite{kn:Green} and the so-called quasi-convexity of the universal cover due to Morse \cite{kn:Morse}). Not assuming anything about curvatures, the Axiom of Asymptoticity introduced in \cite[Definition 5.1 ]{kn:Pesin} also guarantees the continuous foliation-property (see \cite[Theorem 6.1]{kn:Pesin}). At the present state of the art the most general result is the following.
Note that the first claim in this proposition holds true if $(M,g)$ is a compact manifold without conjugate points and has bounded asymptote (we recall its definition at the end of Section \ref{sec:Greenss}) since this property implies uniform divergence of geodesic rays (we refer to \cite{kn:Knieper} and \cite{kn:Green}).

\begin{proposition}[{\cite{kn:RuggieroA}}] \label{hor-continuous}
Let $(M,g)$ be a compact manifold without conjugate points. Then geodesic rays diverge in $(\tilde{M}, \tilde{g})$ if and only if the family $\tilde{\mathscr F}^\s\eqdef\{\tilde{\mathscr F}^\s (\bar\theta)\}_{\bar\theta}$ forms a continuous foliation of $T_{1}\tilde{M}$
(and the latter holds true if and only if $\tilde{\mathscr F}^\u \eqdef\{\tilde{\mathscr F}^\u (\bar\theta)\}_{\bar\theta}$ forms a continuous foliation).
Moreover, both foliations are invariant by the action of the geodesic flow. 

In particular, if $(M,g)$ is a compact surface without conjugate points then the above families form continuous foliations which are invariant by the geodesic flow.
\end{proposition}

The projections of the sets $\tilde{\mathscr F}^\s (\bar\theta)$ and $\tilde{\mathscr F}^\u (\bar\theta)$ by the natural covering map $\bar{\pi}\colon T_{1}\tilde{M} \to T^1M$
give rise to sets ${\mathscr F}^\s (\theta)$ and ${\mathscr F}^\u (\theta)$ which we call \emph{stable} and \emph{unstable foliations}, respectively (these adjectives will be justified by Theorem \ref{thmy:manifolds}).
 In particular, for every $\theta\in T^1M$ and every $t\in\bR$ we have 
\begin{equation}\label{eq:invariance}
	\phi_t(\mathscr F^\s (\theta))
	= \mathscr F^\s (\phi_t(\theta))\\
	\quad\text{ and }\quad
	\phi_t({\mathscr F}^\u (\theta) )
	= \mathscr F^\u (\phi_t(\theta)).
\end{equation}
The collections  $\tilde{\mathscr F}^\s $, $\tilde{\mathscr F}^\u $ are continuous foliations if and only if the families of sets ${\mathscr F}^\s\eqdef\{{\mathscr F}^\s (\theta)\}_\theta$ and ${\mathscr F}^\u\eqdef\{{\mathscr F}^\u (\theta)\}_\theta$ define continuous foliations, respectively.

Finally, let us also define the \emph{center stable} and the \emph{center unstable sets} by
\[
	\tilde{\mathscr F}^\cs(\bar\theta)
	\eqdef \bigcup_{t\in\bR}\phi_t(\tilde{\mathscr F}^\s(\bar\theta))
	\quad\text{ and }\quad
	\tilde{\mathscr F}^\cu(\bar\theta)
	\eqdef \bigcup_{t\in\bR}\phi_t(\tilde{\mathscr F}^\u(\bar\theta)),
\]
respectively. The sets $\tilde{\mathscr F}^\cs(\bar\theta)$ and $\tilde{\mathscr F}^\cu(\bar\theta)$ project to analogously defined sets ${\mathscr F}^\cs(\theta)$ and ${\mathscr F}^\cs(\theta)$, respectively.

One key concept to several topological properties is the following one coined by Eberlein in \cite{kn:Eberlein}. A complete simply connected Riemannian manifold $(M,g)$ is a \emph{uniform visibility manifold} if it has no conjugate points and if for every $\varepsilon>0$ there exists $r=r(\varepsilon)>0$ such that for every $p,x,y\in M$, if the distance between $p$ and the geodesic segment $[x,y]$ is greater than $r$, then the angle at $p$ formed by the geodesic segments $[p,x]$ and $[p,y]$ is less than $\varepsilon$. 

Any compact manifold with negative sectional curvature is a uniform visibility manifold. Moreover, if $(M,g)$ is a compact uniform visibility manifold and $h$ is any other metric on $M$ without conjugate points, then $(M,h)$ is also a uniform visibility manifold \cite{kn:Eberlein}. Hence, in particular, as every compact surface of genus greater than one admits some metric with negative sectional curvature, every compact surface $(M,g)$ without conjugate points is a uniform visibility manifold.

\begin{theorem}\label{teo:2drei}
Let $(M,g)$ be a compact surface without conjugate points.
\begin{itemize}
\item[(1)] The foliations ${\mathscr F}^\s$, ${\mathscr F}^\u$ are minimal.
\item[(2)] The geodesic flow is topologically mixing.
\item[(3)] The geodesic flow has a local product structure in the sense that every two points $(p,v),(q,w)\in T_1\tilde M$, $(q,w)\ne(p,-v)$, are \emph{heteroclinically related}, that is, we have
\[
	\tilde{\mathscr F}^\cs(p,v)\cap\tilde{\mathscr F}^\cu(q,w)\ne \emptyset,\quad
	\tilde{\mathscr F}^\cs(q,w)\cap\tilde{\mathscr F}^\cu(p,v)\ne \emptyset.
\]
\end{itemize}
\end{theorem}	

Let us sketch the ingredients of the proof of Theorem \ref{teo:2drei}. By \cite[Theorem 4.5]{Ebe:77}, for a compact uniform visibility surface the horocyle flow on $T^1M$ is minimal (every orbit is dense). This immediately implies that the foliations $ {\mathscr F}^\s,{\mathscr F}^\u$ both are minimal (every leaf is dense), proving (1). The mixing property is an immediate consequence of minimality of the foliations.
Given $(p,v),(q,w)\in T_1\tilde M$, $(q,w)\ne(p,-v)$, there exists a geodesic $\beta$ such that $\beta(\infty)=\gamma_{(p,v)}(\infty)$ and $\beta(-\infty)=\gamma_{(q,w)}(-\infty)$. There is a unique real parameter $t$ such that $\beta(t)\in H^+(p,v)$ and hence $\dot\beta(t)\in\tilde{\mathscr F}^\s(p,v)$. Moreover, there is a unique real parameter $s$ such that $\gamma_{(q,w)}(s)\in H^-(\beta(t))$ and hence $\beta(t)\in\tilde{\mathscr F}^{\cu}(q,w)$, proving (3).

\begin{remark}[Hyperbolic subsets]\label{rem:hyp}
An invariant set $Z\subset T^1M$ is \emph{hyperbolic} (with respect to the geodesic  flow $\Phi$) if there exist constants $C>0$, $\lambda>0$ and for every $\theta\in Z$ there exist subspaces $E^{\s}(\theta)$ and $E^\u(\theta)$ so that  $E^\s(\theta)\oplus E^\u(\theta)\oplus X(\theta)=T_\theta T^1M$, where $X(\theta)$ here is the subspace tangent to the flow, for every $t\in\bR$ we have $D\phi_t(E^\ast(\theta))=E^\ast(\phi_t(\theta))$, $\ast\in\{\s,\u\}$, and for every $t\ge0$, $\xi\in E^\s(\theta)$, $\eta\in E^\u(\theta)$ we have
\[
	\lVert D\phi_t(\xi) \rVert \le Ce^{-\lambda t}\lVert \xi\rVert,\quad
	\lVert D\phi_{-t}(\eta) \rVert \le Ce^{-\lambda t}\lVert \eta\rVert.
\]
One key feature of a compact hyperbolic set $Z$ is that for every $\theta\in Z$ there exist invariant submanifolds $\mathscr W^{\s}(\theta)$ and $\mathscr W^{\u}(\theta)$ which are stable and unstable sets and at $\theta$ are tangent to the subspaces $E^{\s}(\theta)$ and $E^\u(\theta)$, respectively. The geodesic flow is an \emph{Anosov flow} if $T^1M$ is hyperbolic. It is then an immediate consequence that $\mathscr F^{\s}(\theta)$ and $\mathscr F^{\u}(\theta)$ coincide with the stable and unstable submanifolds $\mathscr W^{\s}(\theta)$ and $\mathscr W^{\u}(\theta)$, respectively, at every point $\theta\in Z$.  
\end{remark}

\subsection{Green subspaces}\label{sec:Greenss}

When studying weaker types of hyperbolicity, it is natural to look for subbundles which are invariant  under the action of the linearization of the flow. Green \cite{kn:Green2} identifies a distinguished family of Jacobi fields defined in any geodesic without conjugate points, which is defined as follows.

For $\theta = (p,v)$, let $N_{\theta}\subset  T_{p}M$ denote the set of vectors that are orthogonal to $v$. Take $\xi \in N_{\theta}$, and let $J_{\xi,T}$ be the Jacobi field of $\gamma_{\theta}$ given by the initial conditions
\[
	J_{\xi,T}(0) = \xi, \quad J_{\xi,T}(T)=0.
\]
By \cite{kn:Green2}, for every $t \in \mathbb{R}$ the limit
\[
	J^\s _{\xi}(t)\eqdef \lim_{T \to  \infty}J_{\xi,T}(t)
\]	
exists (and is a Jacobi field satisfying $J^\s_{\xi}(0)=\xi$). The limit is called \emph{stable Green Jacobi field}.
Analogously the \emph{unstable Green Jacobi field} is defined as the limit
\[
	J^\u _{\xi}(t)
	\eqdef  \lim_{T \to  -\infty}J_{\xi,T}(t).
\]	
Moreover, $J^\s _{\xi}(t)$ and $J^\u _{\xi}(t)$ are always orthogonal to $\dot\gamma_{\theta}(t)$ and never vanish.
The collection of initial conditions
$$
	G^\s (\theta)
	\eqdef  \bigcup_{\xi \in N_{\theta}} \{ (J^\s _{\xi}(0), {J^\s _{\xi}}'(0)) \}
	\quad\text{ and }\quad
	G^\u (\theta)
	\eqdef  \bigcup_{\xi \in N_{\theta}}\{ (J^\u _{\xi}(0), {J^\s _{\xi}}'(0)) \}
$$
are called the \emph{stable Green subspace}  and the \emph{unstable Green subspace} at $\theta$, respectively. Both subspaces are Lagrangian subspaces with respect to the canonical two-form of the geodesic flow) restricted to $N_{\theta}$ and the hence defined vector bundles are invariant under the action of the differential of the geodesic flow:
\begin{equation}\label{eq:invGre}
	D\phi_t(G^\ast(\theta))
	= G^\ast(\phi_t(\theta)),
	\quad\ast\in\{\s,\u\}.
\end{equation}

The above construction can be carried over to the universal cover $\tilde M$ and its tangent space. In particular, for every $\bar\theta\in T_1\tilde M$ one can construct stable and unstable Green subspaces $\tilde G^\s(\bar\theta)$ and $\tilde G^\u(\bar\theta)$, respectively.

Below, we will study the case when stable and unstable Green bundles both are continuous. Note that when $(M,g)$ has this property then in the language of \cite{kn:Eschenburg,kn:Knieper} this manifold has \emph{continuous asymptote}.

\begin{remark}[Green subbundles for hyperbolic subsets]\label{rem:GreHyp}
	Given a hyperbolic compact invariant set $Z\subset T^1M$, for every $\theta\in Z$ the stable and unstable Green subspaces at $\theta$ coincide with the usual stable and unstable subspaces of the dynamics, respectively. Moreover, in this case, for every $\theta\in Z$ the stable and unstable submanifolds of the dynamics coincide with the sets ${\mathscr F}^\s (\theta)$ and ${\mathscr F}^\u (\theta)$, respectively, and hence at every point of these submanifolds the Green Jacobi fields are tangent to them: For every $\eta\in {\mathscr F}^\s (\theta)$ it holds that $G^\s(\eta)$ is tangent to ${\mathscr F}^\s (\theta)$. For every $\eta\in {\mathscr F}^\u (\theta)$ the space $G^\u(\eta)$ is tangent to ${\mathscr F}^\u (\theta)$. 
\end{remark}

\begin{remark}In the general case, Green subspaces may not be tangent to the un-/stable sets everywhere. Indeed, an example due to Ballmann et al. \cite{kn:BBB} shows that there exists compact surfaces without conjugate points where un-/stable Green subspaces do not depend continuously on $\theta$, whereas the collections $\{{\mathscr F}^\s(\theta)\}_\theta$ and $\{{\mathscr F}^\u(\theta)\}_\theta$ are always continuous foliations. 
\end{remark}

Without any further assumption on the dynamics of the geodesic flow, it is difficult to characterize un-/stable Green Jacobi fields since they might have unpredictable asymptotic behavior. When $(M,g)$ has nonpositive curvature, the norm of Jacobi fields is convex and therefore a stable Green Jacobi field $J(t)$ is characterized by the existence of a constant $C>0$ such that  $\sup_{t\ge0}\lVert J(t) \rVert \leq C$. The analogous property holds for an unstable Green Jacobi field with $t \leq 0$.

Perhaps the more general sufficient criterion to characterize an un-/stable Green Jacobi field is the following (the proof follows essentially from the divergence of radial Jacobi field). We call a Jacobi field \emph{radial} if $J(t)=0$ for some $t$. We say that radial Jacobi fields \emph{diverge uniformly} if for any positive numbers $c$ and $a$ there exists $T=T(c,a)>0$ such that every nontrivial radial Jacobi field $J$ with $J(0)=0$ satisfies $\lVert J(t)\rVert\ge c$ and $\lVert J'(t)\rVert\ge a$ for every $t\ge T$. See also \cite[Proposition 2.9]{kn:Eberlein1} or \cite[Chapter 3.2]{kn:Ruggieroensaios}.

\begin{lemma} \label{Green-surfaces}
Let $(M,g)$ be a compact manifold without conjugate points.
\begin{enumerate}
\item[(1)] Any orthogonal Jacobi field $J(t)$ which satisfies $\inf_{t>0} \lVert J(t) \rVert =0$ is a stable Green Jacobi field. 
\item[(2)] Suppose that the radial Jacobi fields of $(M,g)$ diverge uniformly. If a orthogonal Jacobi field $J(t)$ satisfies
$\inf_{t >0} \lVert J(t) \rVert  \leq C$ for some $C>0$ then it is a stable Green Jacobi field. 
\item[(3)] If $(M,g)$ is a compact surface without conjugate points then radial Jacobi fields diverge uniformly and therefore item (2) applies.
\end{enumerate}
The analogous statements hold true for unstable Green Jacobi fields.
\end{lemma}

To fix notation, let us recall some further classifications of manifolds according to their growth behavior of stable Green Jacobi fields (for unstable Green Jacobi fields analogous conditions are put).  A manifold without conjugate points has \emph{bounded asymptote} if there exists $C>0$ such that every stable Green Jacobi field $J$ satisfies $\sup_{t\ge0}\lVert J(t)\rVert\le C\lVert J(0)\rVert$. A manifold has \emph{no focal points} if the norm of any stable Green Jacobi field is always nonincreasing.  
Further, observe that if $M$ has nonpositive curvature then the norm of any stable Green Jacobi field is always a nonincreasing convex function. If $M$ has negative curvature then any stable Green Jacobi field has a norm which decays exponentially. The following implications hold true:
\[
\begin{array}{cccccc}
  &\text{nonpositive curvature} &\Rightarrow   &\text{no focal points}&\Rightarrow &\text{no conjugate points}.
\end{array}
\]
Note also (for example \cite[5.3 Satz]{kn:Knieper}) that for a manifold without conjugate points
\[
\begin{array}{cccc}
  &\text{bounded asymptote}&\Rightarrow   &
  \text{continuous un-/stable Green bundles}\\
  &&&\text{(that is, continuous asymptote)}.
\end{array}
\]

\section{Strips and their relation with Green subspaces} \label{expansivepoints}

In this section, in addition to our Standing Assumption, we assume that $(M,g)$ is a compact surface of genus greater than one.

We start by defining a strip in the universal covering.

\begin{definition}
Given $\bar{\theta} \subset T_{1}\tilde{M}$ the \emph{strip} $ S(\bar{\theta})\subset\tilde{M}$ is the maximal set of geodesics that are bi-asymptotic to $\gamma_{\bar{\theta}}$. 
\end{definition}

The following statement is essentially due to Morse \cite{kn:Morse} and recollects  properties of a strip. Recall the definition of the Busemann flow $\sigma^{\bar{\theta}}_t$ in Section \ref{sec:21}.

\begin{lemma} \label{strips}
For every $\bar\theta\in T_1\tilde M$, it holds
\[
	S(\bar{\theta})
	= \bigcup_{t\in \mathbb{R}} \sigma^{\bar{\theta}}_t( I(\bar{\theta})) ,
	\quad\text{ where }\quad
	 I(\bar{\theta}) 
	\eqdef H^{+}(\bar{\theta}) \cap H^{-}(\bar{\theta}) .
\]	
Moreover, $I(\bar\theta)$ is the arc of a continuous simple curve $\bar c_{\bar{\theta}}\colon [a,b] \to  I(\bar{\theta})$ and $S(\bar\theta)$ is foliated by geodesics which all are bi-asymptotic to $\gamma_{\bar{\theta}}$. 

	If $\bar{\theta} \in T_{1}\tilde{M}$ is the lift of a periodic vector $\theta \in T^1M$, then $S(\bar{\theta})$ is foliated by lifts of periodic geodesics which all are in the same homotopy class of $\gamma_{\theta}$ and which all have the same period.

There exists $Q=Q(M)>0$ such that the Hausdorff distance between any two bi-asymptotic geodesics in $\tilde M$ is bounded from above by $Q$.
\end{lemma}

If the surface has no focal point and, in the notation in Lemma \ref{strips}, if $[a,b]$ is not just one point then the curve $I(\bar\theta)$ is a geodesic and the strip $S(\bar{\theta})$ is \emph{flat}, that is, isometric to $[a,b]\times \mathbb{R}$
endowed with the Euclidean metric for suitably chosen $a<b$ (see the ``flat strip theorem'', \cite[Proposition 5.1]{kn:EO} or \cite[Theorem 7.3]{kn:Pesin}). In general, however, the geometry of a strip might be quite different from a flat object. There are examples of surfaces without conjugate points and with nonflat strips  \cite{kn:Burns}. 

Lemma \ref{strips} justifies the term strip to designate $S(\bar{\theta})$. Note that $I(\bar\theta)$ can contain just a single point, as it is, for example, in the case of negative curvature for any $\bar\theta$. 

\begin{definition}
We say that $S(\bar{\theta})$ is \emph{nontrivial} if $I(\bar\theta)$ is not a single point, otherwise $ S(\bar\theta)$ is \emph{trivial} and in this case we call $\bar\theta$ an \emph{expansive} point.
\end{definition}

Lemma \ref{strips} immediately implies the following.

\begin{corollary} \label{nontransversal}
It holds that
\[
	\bar\cS(\bar\theta)
	\eqdef \bigcup_{t\in\bR}\phi_t(\bar\cI(\bar\theta)),
	\quad\text{ where }\quad
	\bar\cI(\bar\theta)  
	\eqdef \tilde{\mathscr F}^\s (\bar\theta) \cap \tilde{\mathscr F}^\u (\bar\theta),
\]
is a lift of $I(\bar\theta)$ to $T_1\tilde M$. Moreover, $ S(\bar \theta)$ nontrivial if and only if there exists a continuous simple
curve $\bar c_{\bar{\theta}} \colon [0,1] \to T_1\tilde M$ such that 
$\bar c_{\bar{\theta}}([0,1])=  \bar\cI(\bar\theta) $.
\end{corollary}

By Corollary \ref{nontransversal}, the existence of nontrivial strips is equivalent to the existence of (topologically) nontranversal intersections between stable and unstable leaves in $T_1\tilde M$. Let $\cS(\theta) \subset T^1M$ be the image of $\bar\cS(\bar{\theta})$ by the natural projection from $T_{1}\tilde{M}$ to $T^1M$. We shall as well refer to $\cS(\theta)$ as a \emph{strip}.

\begin{definition}
	We call $\theta\in T^1M$ a \emph{generalized rank one vector} if $G^\s(\theta)\not= G^\u(\theta)$. We denote by
\[
	\cR_1
	\eqdef \{\theta\in T^1M\colon G^\s(\theta)\not=G^\u(\theta)\}
\]
the set of all generalized rank one vectors. We denote by
\[
	\cR_0
	\eqdef \{\theta\in T^1M\colon \cS(\theta)\text{ is trivial}\}
	= \{\theta\in T^1M\colon \cF^\s(\theta)\cap\cF^\u(\theta)=\{\theta\}\}
\]
the \emph{set of expansive vectors}.
\end{definition}

The set of expansive points was also studied in \cite[2.1.4]{CliKniWar:}.

It holds $\cR_1\subset\cR_0$. Note that, by invariance of the Green bundles \eqref{eq:invGre} and by \eqref{eq:invariance},  the sets $\cR_1$ and $\cR_0$ both are invariant under the geodesic flow. Moreover, assuming continuity of Green bundles, both sets are open.  Note that if $(M,g)$ is a compact surface without focal points, then $\cR_1$ is just the set of rank one vectors.

\begin{proof}[Proof of Theorem \ref{thmy:manifolds}]
By Knieper \cite[Theorem 3.8]{kn:Knieper},  the stable and the unstable Green Jacobi fields are integrable vector fields, respectively.  Hence, there exist continuous foliations $\cG^\s$ and $\cG^\u$ of $T^1M$ by $C^1$ curves which, using invariance of the Green bundles \eqref{eq:invGre}, are invariant in the sense that for every $\theta\in T^1M$ and every $t\in\bR$ there hold
\[
	\phi_t(\cG^\ast(\theta))=\cG^\ast(\phi_t(\theta))
\]
and
\[
	T_\eta\cG^\ast(\theta)=\cG^\ast(\eta)
	\quad\text{ for every }\eta\in \cG^\ast(\theta),
\]
for $\ast\in\{\s,\u\}$, respectively.

By \cite[Theorem A]{BarRug:07}, the center stable and the center unstable foliations $\cF^\cs$ and $\cF^\cu$ are the only continuous invariant codimension-one foliations of the geodesic flow satisfying the hypotheses. Hence, letting
\[
	\cG^\cs(\theta)
	\eqdef \bigcup_{t\in\bR}\phi_t(\cG^\s(\theta))
	\quad\text{ and }\quad
	\cG^\cu(\theta)
	\eqdef \bigcup_{t\in\bR}\phi_t(\cG^\u(\theta)),
\]
it holds either $\cG^\cs=\cF^\cs$ or $\cG^\cs=\cF^\cu$, and analogously for $\cG^\cu$. 
Let $\theta\in T^1M$ be a hyperbolic periodic vectors. Hence $G^\s(\theta)$ is tangent to $\cF^\s(\theta)$  and $G^\u(\theta)$ is tangent to $\cF^\u(\theta)$ (Remark \ref{rem:GreHyp}). Thus, it follows $\cG^\cs(\theta)=\cF^\cs(\theta)$ and $\cG^\cu(\theta)=\cF^\cu(\theta)$ and, by the assumed continuity, they are transverse in an open neighborhood $U$ of the orbit of $\theta$. Hence, the continuity of the foliations $\cF^\cs$ and $\cF^\cu$ implies 
\begin{equation}\label{samefol}
	\cG^\cs(\eta)=\cF^\cs(\eta)
	\quad\text{ and }	\quad
	\cG^\cu(\eta)=\cG^\cu(\eta)
\end{equation}
for any $\eta\in U$. 
Note that transversality is preserved under the flow. By Eberlein  \cite[Theorem 3.7]{kn:Eberlein}, the geodesic flow is topologically transitive. Together with continuity of these foliations, it follows that  \eqref{samefol} holds for all $\eta\in T^1M$.

Hence, we have already shown that each leaf $\cF^\cs(\theta)$ is sub-foliated by the leaves of $\cG^\s$. Given $\theta\in T^1M$, consider any of its lifts $\bar\theta$ to $T_1\tilde M$ and consider the corresponding foliation $\tilde\cG^\cs$ which by analogous arguments coincides with $\tilde\cF^\cs$. Let $\mu\colon T_1\tilde M\to\tilde M$ be the canonical projection. Recall that the projection $\mu(\tilde\cF^\cs(\bar\theta)) =\mu(\bigcup_{t\in\bR}\phi_t(\tilde\cF^\s(\bar\theta)))$ gives rise to the Busemann flow associated to $\bar\theta$, and the leaves $\mu(\phi_t(\tilde\cF^\s(\bar\theta)))$ are just the horospheres $H^+(\gamma_{\bar\theta}(t))$. The projection $\mu(\phi_t(\tilde\cG^\s(\bar\theta)))$ gives rise to a foliation of $\mu(\tilde\cF^\cs)$ which is everywhere orthogonal to the vector field of the Busemann flow.  Recall that the Green subbundle $G^\s$ is orthogonal to the vector field defining the geodesic flow (in fact, everywhere in $T_1\tilde M$, not just in $\tilde\cF^\cs(\bar\theta)$). 
Since the Busemann vector field $-\nabla b^+_{\bar\theta}$ is a Lipschitz continuous vector field, its orthogonal $(-\nabla b^+_{\bar\theta})^\perp$ inherits this Lipschitz regularity. Hence, the foliations $\{\mu(\phi_t(\tilde\cG^\s(\bar\theta)))\}_t$ and $\{H^+(\gamma_{\bar\theta}(t)\}_t$, being tangent to  $(-\nabla b^+_{\bar\theta})^\perp$, must coincide. This implies $\cF^\s(\bar\eta)=\cG^\s(\bar\eta)$ for every $\eta\in\cF^\cs(\bar\theta)$, which implies $\cF^\s=\cG^\s$. 
This proves item (1) and item (2).

As we assume that both Green bundles vary continuously, given a periodic hyperbolic vector $\theta\in T^1M$,  there is an open set $U\subset T^1M$ containing the orbit of $\theta$ such that  $G^\s(\eta)$ and $G^\u(\eta)$ are linearly independent for every $\eta\in U$. Again using  transitivity, this proves item (3).

Item (4) will be a consequence of Proposition \ref{pro:exponents2} (3).
\end{proof}

\begin{proposition}\label{prolem:liminf}
	For every $\theta\in\cR_1$ that is forward recurrent (with respect to the geodesic flow), for every $\eta\in\cF^\s(\theta)$ there exists a sequence $t_n\to\infty$ such that it holds
\begin{equation}\label{eq:prolem}
	\lim_{n\to\infty}d_\Sak(\phi_{t_n}(\eta),\phi_{t_n}(\theta))=0
\end{equation}
and $\phi_{t_n}(\theta)\to\theta$ as $n\to\infty$. The analogous statement holds true for $\cF^\u$ as $t\to-\infty$.
\end{proposition}

Property \eqref{eq:prolem} was shown in \cite[Lemma 6.7]{CliKniWar:} for almost every vector $\theta\in T^1M$ (relative to any invariant probability measure giving full measure to $\cR_1$) using properties of generalized rank one vectors and ergodic theory-arguments. Notice that, assuming additionally that $(M,g)$ has no focal points, property \eqref{eq:prolem} is true for every $\theta\in \cR_1$ and moreover it holds convergence as $t\to\infty$.

\begin{proof}[Proof of Proposition \ref{prolem:liminf}]
	As $\theta$ is recurrent, there exists a sequence $t_n\to\infty$ such that $\phi_{t_n}(\theta) \to\theta$ as $n\to\infty$. By contradiction, suppose that there exist $\eta\in\cF^\s(\theta)$ and $a>0$ such that  for all $n$ it holds
\[
	d_\Sak(\phi_{t_n}(\eta),\phi_{t_n}(\theta))
	\ge a.
\]		 
Let $\bar\theta\in T_1\tilde M$ be a lift of $\theta$ and let $\bar\eta\in T_1\tilde M$ be a lift of $\eta$  satisfying $\bar\eta\in \tilde\cF^\s(\tilde\theta)$. 
Then there is $a'>0$ such that the corresponding geodesic curves in $\tilde M$ satisfy $d_{\tilde g}(\gamma_{\bar\theta}(t_n),\gamma_{\bar\eta}(t_n))\ge a'$ for all $n$. Recall that for every $t\in \bR$ the point $\gamma_{\bar\eta}(t)$ belongs to the horosphere $H^+(\gamma_{\bar\theta}(t))$. Since, by hypothesis, $\theta$ is accumulated by $\phi_{t_n}(\theta)$, we can choose covering isometries $T_n\colon \tilde M\to\tilde M$ such that
\[
	\bar\theta_n
	\eqdef \big(T_n(\gamma_{\bar\theta}(t_n)),DT_n(\dot\gamma_{\bar\theta}(t_n))\big)
	\to \bar\theta
\]	
as $n\to\infty$. Up to considering some subsequence, we can assume that the sequence
\[
	\bar\xi_n
	\eqdef \big(T_n(\gamma_{\bar\xi}(t_n)),DT_n(\dot\gamma_{\bar\xi}(t_n))\big)
\]
converges as $n\to\infty$, denote  its limit by $\bar\xi_\infty$.
The geodesics $\gamma_{\bar\theta_n}$ and $\gamma_{\bar\xi_n}$ then satisfy:
\begin{itemize}
\item $\lim_{n\to\infty}\gamma_{\bar\theta_n}(t)=\gamma_{\bar\theta}(t)$ uniformly on compact intervals of $t\in\bR$,
\item There exists $D=D(\bar\xi)>0$ such that
\[
	 d_{\tilde g}(\gamma_{\bar\theta_n}(t),\gamma_{\bar\xi_n}(t))\le D
	\quad\text{ for all }t\in[-t_n,\infty),
\]
\item $d_{\tilde g}(\gamma_{\bar\theta_n}(0),\gamma_{\bar\xi_n}(0))\ge a'$ for all $n$,
\end{itemize}
The limiting geodesic $\gamma_{\bar\xi_\infty}$ hence satisfies:
\begin{itemize}
\item $d_{\tilde g}(\gamma_{\bar\theta}(t),\gamma_{\bar\xi_\infty}(t))\le D$ for all $t\in\bR$,
\item $d_{\tilde g}(\gamma_{\bar\theta}(0),\gamma_{\bar\xi_\infty}(0))\ge a'$,
\item $\bar\xi_\infty\in\tilde\cF^\s(\bar\theta)$.
\end{itemize}
Indeed, the latter property is a consequence of the continuity of the stable foliation and the fact that $\bar\xi_n\in\tilde\cF^\s(\bar\theta_n)$ for all $n$.
Thus, the geodesics $\gamma_{\bar\xi_\infty}$ and $\gamma_{\bar\theta}$ are bi-asymptotic and hence they bound a strip of geodesics all being bi-asymptotic. But this contradicts that $\bar\theta$ is the lift of a vector in $\cR_1$.
\end{proof}

Note that if $\theta$ is contained in a hyperbolic invariant set (recall Remark \ref{rem:hyp}; in particular this holds if $\theta$ is a hyperbolic periodic point), then the sets $\mathscr F^{\s}(\theta)$ and $\mathscr F^{\u}(\theta)$ are just the stable and unstable submanifolds at $\theta$ and hence at $\theta$ they are transverse and tangent to the stable and unstable Green subspaces, respectively. The following result states that $\mathscr F^\s(\iota)$ and $\mathscr F^\u(\iota)$ are also transverse as $\iota$ varies along $\mathscr F^\s(\theta)$  (analogously for $\mathscr F^\u(\theta)$). 

\begin{corollary}\label{cor:transverse}
	For every $\theta\in \cR_1$ that is forward recurrent (with respect to the geodesic flow) it holds $\cF^\s(\theta)\subset\cR_1$. In particular, it holds
\[
	\cF^\s(\eta)\cap\cF^\u(\eta)
	=\{\eta\}
	\quad\text{ for all }\quad
	\eta\in\cF^\s(\theta).
\]	
The analogous statements hold true for $\cF^\u(\theta)$.
\end{corollary}

\begin{proof}
	Given $\theta\in\cR_1$ forward recurrent, by Theorem \ref{thmy:manifolds} (3), there exists an open set $U\subset \cR_1$ containing $\theta$. Since $\theta$ is forward recurrent, by Proposition \ref{prolem:liminf} for every $\eta\in\cF^\s(\theta)$ there is a sequence $t_n\to\infty$ such that  $\phi_{t_n}(\eta)\in U$. As the Green bundles are invariant and transversality is preserved under the application of $D\phi_t$, it follows that $G^\s(\eta)$ and $G^\u(\eta)$ are transverse,  and hence $\eta\in\cR_1$. This together with $\cR_1\subset\cR_0$ implies the claim.
\end{proof}

In Section \ref{sec:quotient}, we will use the above results to construct a basis for the quotient topology.

\section{The quotient flow: definition and properties}\label{sec:quotient}

\subsection{Quotient space and the model flow}

Analogously to \cite[Section 4]{GelRug:19}, we say that two points $\theta,\eta\in T^1M$ are \emph{related} $\theta\sim\eta$ if and only if 
\begin{itemize}
\item $\eta\in{\mathscr F}^\s(\theta)$,
\item if $\bar\theta$ is any lift of $\theta$ and $\bar\eta$ is any lift of $\eta$ to $T_1\tilde M$ satisfying $\bar\eta\in{\tilde {\mathscr F}}^\s(\bar\theta)$, then the geodesics $\gamma_{\bar\theta}$ and $\gamma_{\bar\eta}$ are bi-asymptotic.
\end{itemize}
The above relation indeed defines an equivalence relation on $T^1M$.
Given $\theta\in T^1M$, denote by $[\theta]$ the equivalence class containing $\theta$. Denote by $X\eqdef T^1M/_\sim$ the set of all equivalence classes and equip it with the quotient topology. Denote by $\chi\colon T^1M\to X$,  $\chi(\theta)\eqdef [\theta]$, the quotient map. We consider the flow $\Psi=(\psi_t)_{t\in\bR}$, $\Psi\colon \bR\times X\to X$ defined by $\psi_t=\Psi(t,\cdot)$ as
\[
	\psi_t([\theta])
	\eqdef [\phi_t(\theta)].
\] 
This quotient flow is continuous in the quotient topology generated by the topology in $T^1M$. By the very definition of the flows and because the geodesic flow preserves the foliations ${\mathscr F}^\s$ and ${\mathscr F}^\u$ (compare \eqref{eq:invariance}), $\Psi$ is a time-preserving factor of the geodesic flow $\Phi$ by means of $\chi$, that is, for every $t\in\bR$ it holds
\begin{equation}\label{eq:semiconju}
	\chi\circ\phi_t 
	=\psi_t\circ\chi.
\end{equation}

The above defined equivalence relation on $T^1M$ with quotient map $\chi$ naturally induces an equivalence relation in $T_1\tilde M$. We denote by $[\bar\theta]$ the corresponding equivalence class of $\bar\theta\in T_1\tilde M$, by $\bar X$ the set of all equivalence classes, by $\bar\chi\colon T_1\tilde M\to\bar X$ the quotient map, and by $\bar\Psi\colon \bR\times\bar X\to\bar X$ the corresponding quotient flow. Denote by $\bar\Pi\colon\bar X\to X$ the corresponding canonical projection.

The following result is immediate.

\begin{lemma}\label{lem:lochom}
	For every $\theta\in\cR_1$ there exists an open set $U\subset \cR_1$ of $\theta$ such that $\chi|_U\colon U\to\chi(U)$  is a local homeomorphism.
\end{lemma}

\subsection*{Notations}
The following diagram summarizes our setting (compare \eqref{eq:defmfds} for the definition of the stable and unstable sets for the flows $\Psi$ and $\bar\Psi$).
\begin{equation*}
  \!\begin{aligned}
	M 
&\overset{\pi}{\longleftarrow}
	\tilde M,
	\quad I(\bar\theta)\eqdef H^+(\bar\theta)\cap H^-(\bar\theta) 
\\
\theta=(p,v)\in T^1M 			
&\overset{\bar\pi}{\longleftarrow}	\bar\theta=(p,v)\in T_1\tilde M	
\\
\mathscr F^\ast(\theta) 
&\overset{\bar\pi}{\longleftarrow}
	 \tilde{\mathscr F}^\ast(\bar\theta),	\quad \ast=\s,\u,\cs,\cu
\\	
\cI(\theta)
&\overset{\bar\pi}{\longleftarrow}\bar\cI(\bar\theta)\eqdef\tilde\cF^\s(\bar\theta)			\cap\tilde\cF^\u(\bar\theta)
\\
	\phi_t\colon T^1M \to T^1M		
&	\overset{\bar\pi}{\longleftarrow}	
	\phi_t\colon T_1\tilde M \to T_1\tilde M\\
       	\downarrow \chi
&
	\quad\quad\quad\quad\downarrow \bar\chi 
\\	
        \psi_t\colon X \to X	
&	\overset{\bar\Pi}{\longleftarrow}	
	\bar\psi_t\colon \bar X\to \bar X
\\
W^\ast([\theta]) 		
&\overset{\bar\Pi}{\longleftarrow}	
\tilde W^\ast([\bar\theta]),
	\quad \ast=\ss,\uu,\cs,\cu
  \end{aligned}
\end{equation*}

The following two results establish the essential properties of the quotient space and of the dynamical properties of the quotient flow.

\begin{theorem}\label{teo:quotient}
	Let $(M,g)$ be a compact surface without conjugate points and continuous  stable and unstable Green bundles. Then the quotient space $X$ is a compact topological $3$-manifold. In particular, $X$ admits a smooth $3$-dimensional structure where the quotient flow $\Psi$ is continuous. 
\end{theorem}

The proof of Theorem \ref{teo:quotient} will be sketched in Section \ref{secteo:quotient}.
In the following, let us fix some metric $d$ on $X$ which is induced by a Riemannian metric. We will recall expansiveness and local product structure in Sections \ref{sec:exp} and \ref{sec:locpro}, respectively.

\begin{theorem}\label{teo:quotientflow}
	Let $(M,g)$ be a compact surface without conjugate points and continuous  stable and unstable Green bundles. Then the quotient flow $\Psi$ is expansive, topologically mixing, and has a local product structure.
\end{theorem}

The proof of Theorem \ref{teo:quotientflow} will be completed in Section \ref{sec:poooof}.

\subsection{Proof of Theorem \ref{teo:quotient}}\label{secteo:quotient}

The proof is analogous to \cite[Theorem 4.3]{GelRug:19}. We only sketch it, indicating the differences.

As in \cite{GelRug:19}, given a vector $\bar\theta\in T_1\tilde M$ and $\varepsilon>0$ and $\delta>0$ sufficiently small, we choose the local cross-section 
\[
	\Sigma_{\bar\theta}(\varepsilon,\delta)
	\eqdef R((r_0^--\varepsilon,r_0^++\varepsilon)\times(-\delta,\delta)),
\]	 
where $R\colon(r_0^--\varepsilon,r_0^++\varepsilon)\times(-\delta,\delta)\to T_1\tilde M$ is the homeomorphism having the properties
\begin{itemize}
\item $R(0,0)=\bar\theta$;
\item $s\mapsto R(0,s)$, $s\in(-\delta,\delta)$, is the arc length parametrization (in the Sasaki metric) of the $\delta$-tubular neighborhood $V_\delta(\bar\theta)$ of $\bar\theta$ in its vertical fiber;
\item $r\mapsto R(r,0)$, $r\in (r_0^--\varepsilon,r_0^++\varepsilon)$, is the arc length parametrization of the $\varepsilon$-tubular neighborhood of $\bar\cI(\bar\theta)$ in $\tilde {\mathscr F}^\s(\bar\theta)$, with $R(r_0^-,0)$ and $R(r_0^+,0)$ being the endpoints of $\bar\cI(\bar\theta)$;
\item for each $s\in(-\delta,\delta)$, $r\mapsto R(r,s)$ is the arc length parametrization of the continuous curve in $\tilde{\mathscr F}^\s(R(0,s))$.
\end{itemize}
This defines a continuously embedded closed two-dimensional disc $\Sigma_{\bar\theta}(\varepsilon,\delta)$ transverse to the geodesic flow, containing $\bar\cI(\bar\theta)$, and foliated by leaves of $\tilde{\mathscr F}^\s$ (compare also Figure \ref{fig:1}). In the following, we will shortly denote this disk by $\Sigma$.
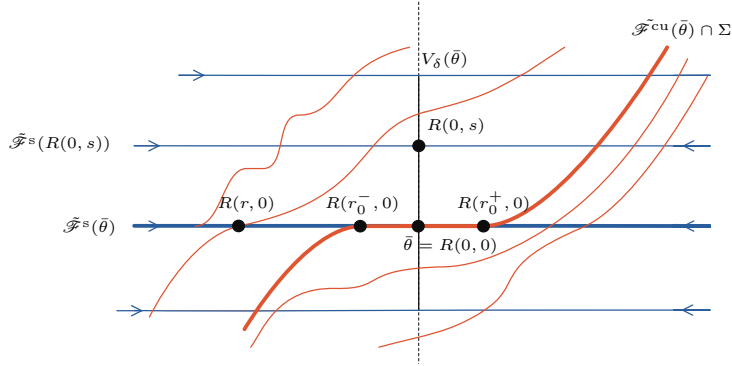
\begin{figure}
\begin{overpic}[scale=.48 
  ]{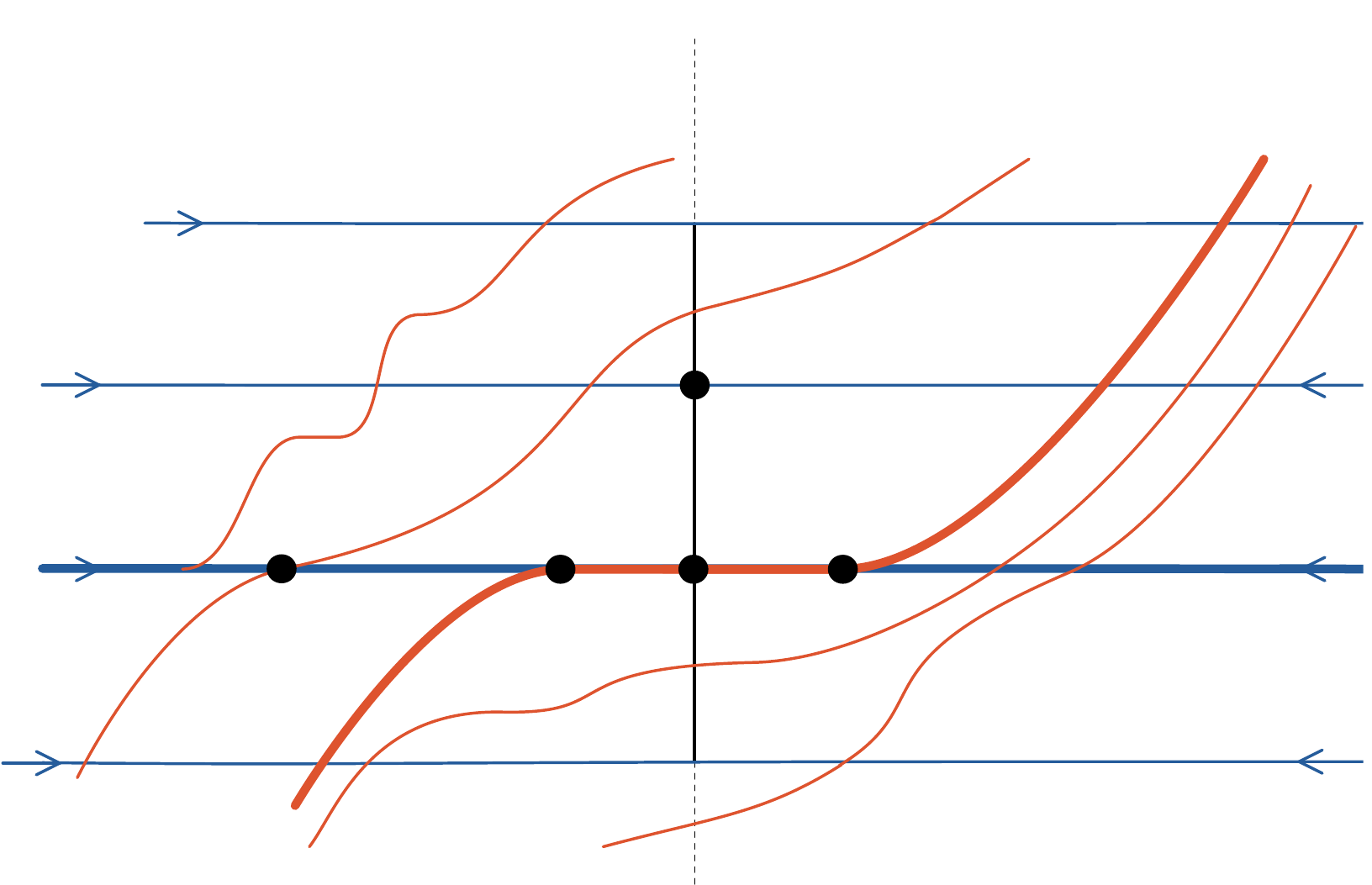}
      \put(35,26){\tiny$R(r_0^-,0)$}
      \put(57,26){\tiny$R(r_0^+,0)$}
      \put(52,39){\tiny$R(0,s)$}
      \put(17,26){\tiny$R(r,0)$}
      \put(-9,22){\tiny$\tilde{\mathscr F^\s}(\bar{\theta})$}
      \put(-18,36){\tiny$\tilde{\mathscr F^\s}(R(0,s))$}
      \put(86,55){\tiny$\tilde{\mathscr F^{\cu}}(\bar{\theta})\cap \Sigma$}
      \put(48,19){\tiny$\bar{\theta}=R(0,0)$}
      \put(51,50){\tiny$V_\delta(\bar\theta)$}
\end{overpic}
\caption{Parametrization of the local cross-section $\Sigma=\Sigma_{\bar\theta}(\varepsilon,\delta)$}
\label{fig:1}
\end{figure}

Given an interval $(a,b)$ and $Y\subset T_1\tilde M$, denote 
\[
	\phi_{(a,b)}(Y)\eqdef\bigcup_{t\in(a,b)}\phi_t(Y).
\]	
For $\tau>0$, consider the open neighborhood of $\bar\theta$ in $T_1\tilde M$ defined by
\[
	B_{\bar\theta}(\varepsilon,\delta,\tau)
	\eqdef \phi_{(-\tau,\tau)}(\Sigma)
\]
and consider  the projection $\Pi_\Sigma\colon B_{\bar\theta}(\varepsilon,\delta,\tau)\to\Sigma$ by the flow $\Phi$. 

Given any strip $\bar\cS$ which intersects $B_{\bar\theta}(\varepsilon,\delta,\tau)$, there exists a vector $\bar\eta\in\Sigma$ such that $\Pi_\Sigma(\bar\cS\cap B_{\bar\theta}(\varepsilon,\delta,\tau))$ is a connected component of $\bar\cI(\bar\eta)$. In particular, if $\bar\cI(\bar\eta)\subset B_{\bar\theta}(\varepsilon,\delta,\tau)$ then $\Pi_\Sigma(\bar\cS)=\bar\cI(\bar\eta)$. 

Given $\bar\eta\in B_{\bar\theta}(\varepsilon,\delta,\tau)$, denote by
\[
	B^\cs_{\bar\eta}(\varepsilon,\delta,\tau)
	\eqdef \tilde{\mathscr F}^\cs(\bar\eta)\cap B_{\bar\theta}(\varepsilon,\delta,\tau),
	\quad
	B^\cu_{\bar\eta}(\varepsilon,\delta,\tau)
	\eqdef \tilde{\mathscr F}^\cu(\bar\eta)\cap B_{\bar\theta}(\varepsilon,\delta,\tau),
\]
the connected components of the intersections of the central stable and the central unstable sets of $\bar\eta$ with $B_{\bar\theta}(\varepsilon,\delta,\tau)$ that contain $\bar\eta$, respectively. Denote
\[
	W^\s_\Sigma(\bar\eta)
	\eqdef \Pi_\Sigma(B_{\bar\eta}^\cs(\varepsilon,\delta,\tau)),
	\quad
	W^\u_\Sigma(\bar\eta)
	\eqdef \Pi_\Sigma(B_{\bar\eta}^\cu(\varepsilon,\delta,\tau)).
\]

Choose $\delta_0=\delta(\bar\theta)>0$ so small that $\tilde\cF^\cu(\bar\theta)$ intersects $ \tilde\cF^\s(R(0,(-\delta_0,\delta_0)))$. Given $\delta\in(0,\delta_0)$, there exists $\varepsilon_0=\varepsilon_0(\bar\theta,\delta)>0$ so that for every $\varepsilon\in(0,\varepsilon_0)$, every $r\in(r_0^--\varepsilon,r_0^-)\cup(r_0^+,r_0^++\varepsilon)$, and every $s\in(-\delta,\delta)$ for points $\bar\eta\eqdef R(0,r)$ and $\bar\xi\eqdef R(0,s)$ the intersection 
\[
	[\bar\xi,\bar\eta]
	\eqdef
	W^\s_\Sigma(\bar\xi)\cap W^\u_\Sigma(\bar\eta)
	\subset \Sigma
\]
is nonempty (though they may be contained in a nontrivial strip).

Let us consider the following open subsets of $\Sigma$ 
\[\begin{split}
	\Sigma^{+,+}
	&\eqdef \{R(r,s)\colon r\in(0,r_0^++\varepsilon),s\in(0,\delta)\},\\
	\Sigma^{+,-}
	&\eqdef \{R(r,s)\colon r\in(0,r_0^++\varepsilon),s\in(-\delta,0)\},\\
	\Sigma^{-,+}
	&\eqdef \{R(r,s)\colon r\in(r_0^--\varepsilon,0),s\in(0,\delta)\},\\
	\Sigma^{-,-}
	&\eqdef \{R(r,s)\colon r\in(r_0^--\varepsilon,0),s\in(-\delta,0)\}.
\end{split}\]
By Corollary \ref{cor:transverse}, given a hyperbolic periodic $\eta\in T^1M$ and  any lift $\bar\eta$, then the leaves $\tilde\cF^\s(\bar\eta)$ and $\tilde\cF^\u(\bar\eta)$ do not intersect any nontrivial strip. By invariance, $\tilde\cF^\cs(\eta)$ and $\tilde\cF^\cu(\eta)$ also do not contain a nontrivial strip. By Theorem \ref{teo:2drei}, the foliations are minimal and hence, in particular, the leaf $\cF^\s(\eta)$ and the leaf $\cF^\u(\theta)$ both are dense in $T^1M$. Therefore  there exist lifts of $\eta$ to $T_1\tilde M$ whose center stable set intersects the sets $\Sigma^{+,+}$ and $\Sigma^{-,-}$ in points $\bar\xi^+$ and $\bar\xi^-$, respectively. Analogously, there exists a lift whose center unstable set intersects $\Sigma^{-,+}$ and $\Sigma^{+,-}$ in points $\bar\eta^+$ and $\bar\eta^-$, respectively.  Note that $\bar\eta^\pm$ and $\bar\xi^\pm$ are expansive points. Moreover the sets $W_\Sigma^\s(\bar\xi^\pm)$ and $W_\Sigma^\u(\bar\eta^\pm)$ are curves which are disjoint from $W_\Sigma^\u(\bar\theta)$. Each of the intersections $[\bar\xi^\pm,\bar\eta^\pm]=W_\Sigma^\s(\bar\xi^\pm) \cap W_\Sigma^\u(\bar\eta^\pm)$ contains a single point and the corresponding arcs bound a region in $\Sigma$ which is homeomorphic to a rectangle whose relative interior contains $\bar\cI(\bar\theta)$. 
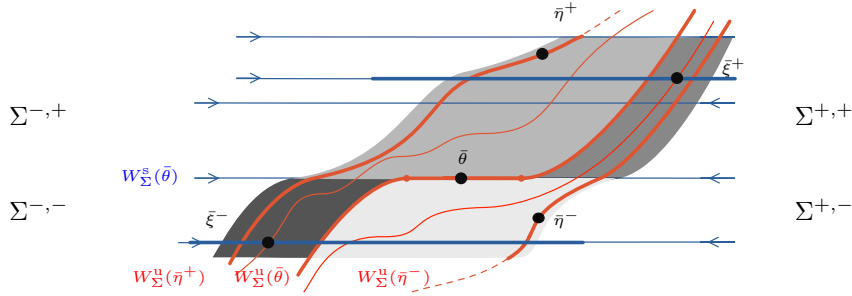
\begin{figure}
\begin{overpic}[scale=.45,%
  ]{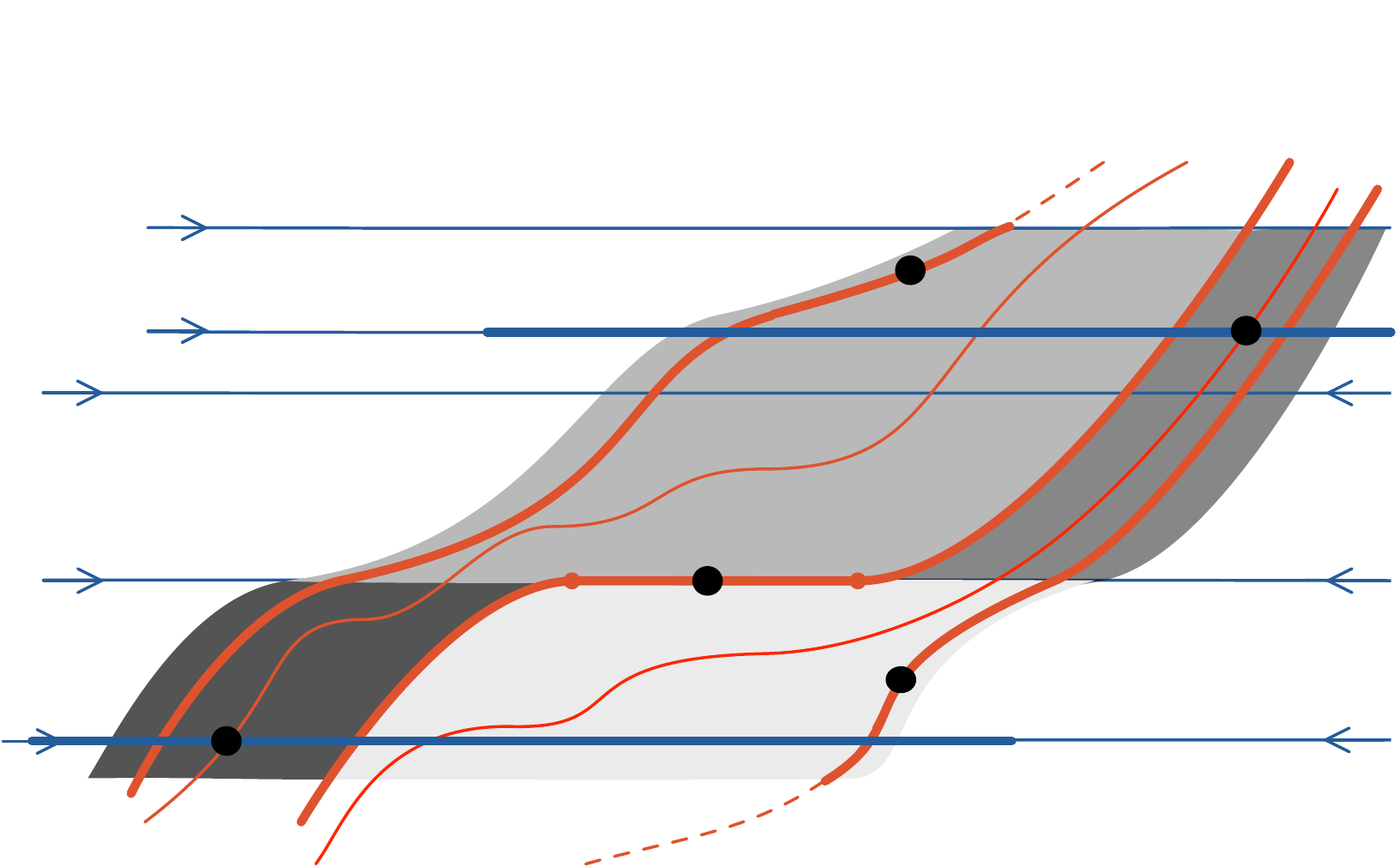}
      \put(5,12){\tiny$\bar\xi^-$}
      \put(67,12){\tiny$\bar\eta^-$}
      \put(97,39){\tiny$\bar\xi^+$}
      \put(-8,2){\tiny\textcolor{red}{$W^\u_\Sigma(\bar\eta^+)$}}
      \put(32,2){\tiny\textcolor{red}{$W^\u_\Sigma(\bar\eta^-)$}}
      \put(10,2){\tiny\textcolor{red}{$W^\u_\Sigma(\bar{\theta})$}}
      \put(-10,20){\tiny\textcolor{blue}{$W^\s_\Sigma(\bar{\theta})$}}
      \put(50,23){\tiny$\bar{\theta}$}
      \put(67,49){\tiny$\bar{\eta}^+$}
      \put(110,30){$\Sigma^{+,+}$}
      \put(-30,30){$\Sigma^{-,+}$}
      \put(-30,13){$\Sigma^{-,-}$}
      \put(110,13){$\Sigma^{+,-}$}
\end{overpic}
\caption{Region defined by expansive points $\bar\eta^\pm$, contained in the region $\Sigma_{\bar{\theta}}(\varepsilon,\delta)$ splits into $\Sigma^+$ and $\Sigma^-$, and containing the open set $U_{\bar\theta}(\varepsilon,\delta,\bar\xi^-,\bar\xi^+,\bar\eta^-,\bar\eta^+)$ (shaded region)}
\label{fig:4}
\end{figure}
Denote by $U=U_{\bar\theta}(\varepsilon,\delta,\bar\xi^-,\bar\xi^+,\bar\eta^-,\bar\eta^+)$ this open region in $\Sigma$ whose boundary is formed by the corresponding arcs contained in $W^\u_\Sigma(\bar\eta^+)$, $W^\s_\Sigma(\bar\xi^+)$, $W^\u_\Sigma(\bar\eta^-)$, and $W^\s_\Sigma(\bar\xi^-)$ (compare Figure \ref{fig:4}).

Following now verbatim arguments in the proofs of \cite[Lemmas 4.6 and 4.7 and Proposition 4.8]{GelRug:19}, we show the following.

\begin{lemma}\label{lem:recregio}
	Given $\bar\theta\in T_1\tilde M$, there exists $\delta_0=\delta_0(\bar\theta)>0$ and for every $\delta\in(0,\delta_0)$ there exists $\varepsilon_0=\varepsilon_0(\bar\theta,\delta)$ so that for every $\varepsilon\in(0,\varepsilon_0)$ there are numbers $\rho^\pm\in(0,\varepsilon)$ and expansive points $\bar\eta^+\in\Sigma^{-,+}$, $\bar\eta^-\in\Sigma^{+,-}$, $\bar\xi^-\in\Sigma^{-,-}$, and $\bar\xi^+\in\Sigma^{+,+}$, where $\Sigma=\Sigma_{\bar\theta}(\varepsilon,\delta)$. Consider the above constructed region $U=U_{\bar\theta}(\varepsilon,\delta,\bar\xi^-,\bar\xi^+,\bar\eta^-,\bar\eta^+)\subset\Sigma$. Then for $\tau>0$ sufficiently small the set
\[
	A
	= A_{\bar\theta}(\tau,\varepsilon,\delta,\bar\xi^-,\bar\xi^+,\bar\eta^-,\bar\eta^+)
	\eqdef \phi_{(-\tau,\tau)}(U)
\]	 
satisfies $\bar\chi^{-1}(\bar\chi(A))=A$.  The collection
\begin{equation}\label{eq:bases}
	\{\bar\chi(A_{\bar\theta}(\tau,\varepsilon,\delta,\bar\xi^-,\bar\xi^+,\bar\eta^-,\bar\eta^+))\}
\end{equation}
provides a basis for the quotient topology of $\bar X$.

Moreover, there exist numbers $a<a'$ and $b<b'$ and a homeomorphism 
\[
	f\colon (a,a')\times (b,b')\times(-\tau,\tau)
	\to \bar\psi_{(-\tau,\tau)}(\bar\chi(U))
	=\bar\chi(A)
\]	
for every $\tau>0$.
In particular, the quotient space $\bar X$ is a topological $3$-manifold.
\end{lemma}

By the above lemma, every point in $\bar X$ has an open neighborhood which is homeomorphic to an open subset of $\bR^3$. Hence $\bar X$ is a topological $3$-manifold. By \cite{Bin:59,Moi:77}, the space $\bar X$ has a smooth structure which is compatible with the quotient topology. Since its quotient $X$ is locally homeomorphic to $\bar X$, the last assertion also extends to this space. This sketches the proof of the theorem.
\qed

\subsection{Expansiveness}\label{sec:exp}

A continuous flow $\Psi=(\psi_t)_t$ on a compact metric space $X$ is \emph{expansive} if for every $\varepsilon>0$ there exists $\delta>0$ with the property that if $x \in X$ and $y \in X$ is a point for which there exists an increasing homeomorphism $\rho\colon \bR \to \bR$ with $\rho(0)=0$ and for every $t \in {\mathbb R}$ satisfying
\[
 	d( \psi_t(x), \psi_{\rho(t)}(y)) \leq \delta,
\]	
then it holds $y=\psi_{t(y)}(x)$ for some $\lvert t(y) \rvert \le \varepsilon$ (see \cite[Theorem 3]{kn:BW72} for equivalent definitions).
One calls $\varepsilon$ a \emph{constant of expansivity}.

\begin{proposition}\label{pro:expansive}
	The quotient flows $\bar\Psi$ and $\Psi$ both are expansive flows.
\end{proposition}

The proof of this proposition goes verbatim to the arguments in \cite[Section 5.1]{GelRug:19}.
 Indeed, observe that the key argument used in \cite{GelRug:19} is that the width of any strip is bounded from above, which is a consequence of the fact that any two bi-asymptotic geodesics in $\tilde M$ stay in uniformly bounded Hausdorff distance from each other. But this latter fact continues to hold true for compact Riemannian surfaces without conjugate points and of genus greater than one (compare Lemma \ref{strips}).

\subsection{Dynamics on stable and unstable sets}

In this section we will study the quotients of the stable and unstable manifolds. 

Given $\theta\in T^1M$ and one of its lifts $\bar\theta\in T_1\tilde M$, define 
\begin{equation}\label{eq:defmfds}\begin{split}
	\tilde W^\ast([\bar\theta])
	&\eqdef\bar\chi(\tilde\cF^\ast(\bar\theta)),
	\quad
	 W^\ast([\theta])
	\eqdef\chi(\cF^\ast(\theta)),
	\quad \ast\in\{\cs,\cu\}\\
	\tilde W^\ss([\bar\theta])
	&\eqdef\bar\chi(\tilde\cF^\s(\bar\theta)),
	\quad
	 W^\ss([\theta])
	\eqdef\chi(\cF^\s(\theta))\\
	\tilde W^\uu([\bar\theta])
	&\eqdef\bar\chi(\tilde\cF^\u(\bar\theta)),
	\quad
	 W^\uu([\theta])
	\eqdef\chi(\cF^\u(\theta)).
\end{split}\end{equation}
The following relations are immediate consequences of the semi-conjugation \eqref{eq:semiconju} and the definitions in \eqref{eq:defmfds} 
\[
	\tilde W^\cs([\bar\theta])
	\eqdef \bigcup_{t\in\bR}\bar\psi_t(\tilde W^\ss([\bar\theta]))
	= \chi(\tilde\cF^\cs(\bar\theta)),
\]
analogously for $\tilde W^\cs([\bar\theta])$. 
Moreover, for $x\eqdef \chi(\theta)=\bar\Pi([\bar\theta])$ it holds
\[
	W^\cs(x)
	\eqdef  \bigcup_{t\in\bR}\psi_t(W^\ss(x))
	= \bar\Pi(\tilde W^\cs([\bar\theta])),
\]
analogously for $W^\cu(x)$.

The  proof of the following result will be completed at the end of this section. In this section we only investigate stable sets, the analogous results hold true for unstable sets.

\begin{proposition}[Uniform contraction in $W^\ss$]\label{pro:uniformcontra}
	For every $D>0$, for every $a>0$ there exists $T>0$ such that  for every $\theta$ for $x\in X$ for all $t\ge T$ it holds that
\[
	d(\psi_t(y),\psi_t(x))\le a
	\quad\text{ for every }y\in W^{\ss}(x)\text{ with }d(y,x)\le D.
\]	
\end{proposition}

Besides the fact that the quotient flow is expansive, Proposition \ref{prolem:liminf} will be a key fact towards proving Proposition \ref{pro:uniformcontra}.
Let us introduce some more notation. Let $\theta\in T^1M$ and $x\eqdef\chi(\theta)$. 
Let 
\[
	c^{\s}_\theta\colon (-\infty,\infty)\to \cF^{\s}(\theta)
\]
be the parametrization of $\cF^\s(\theta)$ by arc length  satisfying $c^\s_\theta(0)=\theta$. Let 
\[
	\cF^{\s}_D(\theta)
	\eqdef c^{\s}_\theta([-D,D]).
\]
Define, as in \eqref{eq:defmfds}, 
\[
	W^{\ss}(x)
	\eqdef \chi(\cF^{\s}(\theta)),\quad\text{ and let }\quad
	W^{\ss}_D(x)
	\eqdef \chi(\cF^{\s}_D(\theta)).
\]
Note that the sets $\cF^{\s}_D(\theta)$ are compact and depend continuously on $\theta\in T^1M$ and $D>0$. By continuity of the quotient map, their quotients $W^{\ss}_D(x)$ also depend continuously on $x=\chi(\theta)$ and $D$. 

Notice that if $\cF^{\s}_D(\theta)$ is contained in a nontrivial equivalence class, then its quotient  collapses to the single point 
\[
	\chi(\cF^{\s}_D(\theta))
	= \{\chi(\theta)\}.
\]
So \emph{a priori} the geometry of such quotient curves can be quite singular. 

Let us argue about the differentiable nature of the above defined foliations. 

\begin{lemma}
	For every $\theta\in\cR_1$, there exists $D>0$ such that $W^\ss_D(\chi(\theta))$ is smooth in $X$. Moreover, for every $\theta\in\cR_1$ that is forward recurrent (with respect to the geodesic flow), $W^\ss(\chi(\theta))$ is smooth everywhere.	\end{lemma}

\begin{proof}
	By Theorem \ref{teo:quotient}, the quotient space $X$ admits a differentiable structure and that we denoted by $d$ a distance induced by a Riemannian metric. By Theorem \ref{thmy:manifolds} (3), $\cR_1$ is an open subset of $T^1M$. By Lemma \ref{lem:lochom}  there exists an open neighborhood $V\subset\cR_1$ containing  $\theta$ so that  $\chi|_V\colon V\to\chi(V)$ is a homeomorphism. 
	Let $\varphi\colon U\to V$, $U\subset\bR^3$ open, be some local chart for $T^1M$. By Theorem \ref{teo:quotient} (1), the set $\cF^\s(\theta)$ is a $C^1$-curve embedded in $T^1M$. Choose $D>0$ such that $\cF^\s_D(\theta)\subset V$ and consider its parametrization $c^\s_\theta\colon[-D,D]\to\cF^\s_D(\theta)$ by arc length. Then $\chi\circ c^\s_D\colon [-D,D]\to \chi(\cF^\s_D(\theta))=W^\ss_D(x)$ is smooth and parametrizes $W^\ss(x)\cap\chi(V)$. Indeed, $\chi\circ\varphi\colon U\to\chi(V)$ is a local chart for $\chi(V)\subset X$ and it holds $(\chi\circ\varphi)^{-1}\circ(\chi\circ c)=\varphi^{-1}\circ c$.
	
	Finally, let $x\eqdef\chi(\theta)$ for $\theta\in\cR_1$ be forward recurrent. By contradiction, suppose that there exists a point $y\in W^\ss(x)$ of nondifferentiability. By the first claim, $y\ne x$. Let $\eta\in\chi^{-1}(y)$ and hence $\eta\in \cF^\s(\theta)$. By hypothesis, the orbit of $\theta$ recurs infinitely often to any neighborhood of $\theta$. Moreover, by Proposition \ref{prolem:liminf}, the Sasaki distance between the images of $\eta$ and $\theta$ under these recurrence time-flow maps converges to zero. Hence, as $\cR_1$ was open, for sufficiently large $t>0$, it holds $\phi_t(\eta)\in \cR_1$. As $\cR_1$ is invariant, it follows $\eta\in\cR_1$. But this gives a contradiction with the first part of the proof.
\end{proof}	  

\begin{lemma}\label{lem:2ooo}
	Let $\theta\in\cR_1$ be forward recurrent (with respect to the geodesic flow) and $x\eqdef \chi(\theta)$. Then for every $y\in W^\ss(x)$ it holds $\liminf_{t\to\infty}d(\psi_t(y),\psi_t(x))=0$.
\end{lemma}

\begin{proof}
	Recall that, by Lemma \ref{lem:lochom}, there exists an open neighborhood $U=U(\theta)$ of $\theta$ such that $\chi|_U\colon U\to\chi(U)$ is a homeomorphism, considering the Sasaki distance $d_\Sak$ in $U$ and the metric $d$ in $\chi(U)$, respectively. Then $\chi(U)$ is an open set. 
 Let $\eta\in \cF^\s(\theta)$ such that $\chi(\eta)=y$. By Proposition \ref{prolem:liminf}, there exists a sequence $t_n\to\infty$ such that $\phi_{t_n}(\theta)\in U$ and $\phi_{t_n}(\theta)\to\theta$ and $d_\Sak(\phi_{t_n}(\xi),\phi_{t_n}(\theta))\to 0$ as $n\to\infty$. 
This implies the claim.
 \end{proof}

The proof of the following result uses strongly the expansivity of the quotient flow.

\begin{proposition}[Pseudo-convexity of orbits]\label{lem:3ooo}
	For every $\varepsilon>0$ there exists $\delta>0$ such that for every $x\in X$,	every $t\ge0$, and every $y\in W^{\ss}(x)$ satisfying
\[
	\max\big\{d(y,x),d(\psi_t(y),\psi_t(x))\big\}\le\delta
\]	
it holds $d(\psi_s(y),\psi_s(x))\le\varepsilon$ for all $s\in[0,t]$. 
\end{proposition}

\begin{proof}
We argue by contradiction. Suppose that there exist $a>0$, sequences of points $\theta_n\in T^1M$, $x_n\eqdef\chi(\theta_n)$, $\eta_n\in \cF^{\s}(\theta_n)$, and $y_n\eqdef\chi(\eta_n)$, and a sequence of times $t_n\to\infty$ as $n\to\infty$ and $T_n\in(0,t_n)$ such that
\[
	d(y_n,x_n)\le\frac1n
	\quad\text{ and }\quad
	d(\psi_{t_n}(y_n),\psi_{t_n}(x_n))\le\frac1n
\]
and 
\[
	d(\psi_{T_n}(y_n),\psi_{T_n}(x_n))\ge a.
\]
From continuity of the flow $\psi_t$, it follows that $T_n\to\infty$ and $t_n-T_n\to\infty$ as $n\to\infty$.

Let $\varepsilon\in(0,a]$ be an expansivity constant for the flow $\psi_t$. 

Recalling that $W^{\ss}(x_n)=\chi(\cF^{\s}(\theta_n))$, let $c_n^\s\colon[0,\delta_n]\to \cF^{\s}(\theta_n)$ be the continuous curve parametrized by arc length and joining $\theta_n$ and $\eta_n$ such that $\chi\circ c_n^\s(0)=x_n$ and $\chi\circ c_n^\s(\delta_n)=y_n$. 
Let $\varepsilon_n\in(0,\delta_n]$ such that 
\[
	\sup_{r\in[0,\varepsilon_n],t\in[0,t_n]}
	d\big(\psi_t(\chi\circ c_n^\s(r)),\psi_t(x)\big)
	= \varepsilon.
\]
As $\chi$ is continuous and the distance restricted to the arc connected sets $W^{\ss}(x_n)$ is continuous, the above supremum is in fact obtained at some $r=\varepsilon_n'\in(0,\varepsilon_n]$, and there exists a sequence $s_n\in[0,t_n]$ such that 
\[
	d(\psi_{s_n}(\chi\circ c_n^\s(\varepsilon_n')),\psi_{s_n}(x))
	=\varepsilon.
\]
Again by continuity of the flow $\psi_t$, it holds $s_n\to\infty$ and $s_n-t_n\to\infty$ as $n\to\infty$.

Consider now the sequences of points
\[
	z_n
	\eqdef \psi_{s_n}(x_n)
	\quad\text{ and }\quad
	w_n
	\eqdef \psi_{s_n}(\chi\circ c_n^\s(\varepsilon_n')).
\]
Notice that they are quotients by $\chi$ of vectors
\[
	\zeta_n
	\eqdef \phi_{s_n}(\theta_n)
	\quad\text{ and }\quad
	\xi_n
	\eqdef \phi_{s_n}(c_n^\s(\varepsilon_n')),
\]
respectively. Note that 
\begin{itemize}
\item $\xi_n\in\cF^{\s}(\zeta_n)$ (by invariance \eqref{eq:invariance}) and hence
\item $w_n\in \chi(\cF^{\s}(\zeta_n))=W^{\ss}(z_n)$ 
\item 
	$d(\psi_t(w_n),\psi_t(z_n))
	\le \varepsilon$
for every $t\in[-s_n,t_n-s_n]$
\item $d(w_n,z_n)=\varepsilon$.
\end{itemize}
Up to passing to some subsequence, we can assume that these sequences converge to limit points $\zeta_\infty=\lim_n\zeta_n$ and $\xi_\infty=\lim_n\xi_n$ and hence we have limit points $z_\infty=\lim_nz_n$ and $w_\infty=\lim_nw_\infty$, respectively. It follows from the continuity of foliations that $\lim_n\cF^{\s}(\zeta_n)=\cF^{\s}(\zeta_\infty)$ and hence $\xi_\infty\in\cF^{\s}(\zeta_\infty)$.
Thus, we obtain
\begin{itemize}
\item $w_\infty\in W^{\ss}(z_\infty)$; moreover
\item $d(\psi_t(z_\infty),\psi_t(w_\infty))\le\varepsilon$ for all $t\in\bR$ and
\item $d(w_\infty,z_\infty)=\varepsilon$.
\end{itemize}
But this contradicts the fact that the flow $\psi_t$ is expansive.
\end{proof}

\begin{corollary}\label{cor:lem3ooo}
	For every $\varepsilon>0$ there exists $\delta>0$ such that for every $\theta\in\cR_1$ forward recurrent (with respect to the geodesic flow), for $x\eqdef\chi(\theta)$ and for every $y\in W^{\ss}(x)$, $d(y,x)\le \delta$ for all $t\ge0$ it holds
\[
	d(\psi_t(y),\psi_t(x))
	\le \varepsilon
\]	
and 
\[
	\lim_{t\to\infty}d(\psi_t(y),\psi_t(x))=0.
\]	
\end{corollary}

\begin{proof}
Lemma \ref{lem:2ooo} together with Proposition \ref{lem:3ooo} implies  the claim. 
\end{proof}

\begin{proof}[Proof of Proposition \ref{pro:uniformcontra}]
We first prove the claim for $x$ such that $x=\chi(\theta)$ where $\theta\in \cR_1$ is forward recurrent (with respect to the geodesic flow).
Let $\varepsilon>0$ be a constant of expansivity for the flow $\psi_t$. Let $\delta=\delta(\varepsilon)>0$ as provided by Proposition \ref{lem:3ooo}. 
We argue again by contradiction. Let $D\in(0,\varepsilon)$. Suppose that there exist $a\in(0,D)$, sequences of forward recurrent vectors $\theta_n\in \cR_1$, points $x_n=\chi(\theta_n)$, and $y_n\in W^{\ss}(x_n)$ satisfying $d(y_n,x_n)\le D$, and a sequence of times $t_n\to\infty$ as $n\to\infty$ such that for every $n\ge1$
\[
	d(\psi_{t_n}(y_n),\psi_{t_n}(x_n))\ge a. 
\]
By Lemma \ref{lem:2ooo}, there exists a sequence $\tau_n$ such that $\tau_n-t_n\to\infty$ and 
\[
	d(\psi_{\tau_n}(y_n),\psi_{\tau_n}(x_n))\le D.
\]	
Hence, the points
 $z_n\eqdef\psi_{t_n}(x_n)$ and $w_n\eqdef\psi_{t_n}(y_n)$ satisfy 
\begin{itemize}
\item $w_n\in W^\ss(z_n)$,
\item $d(\psi_t(w_n),\psi_t(z_n))\le\varepsilon$ for every $t\ge -t_n$, and
\item $d(w_n,z_n)\ge a$.
\end{itemize}
Up to passing to some subsequence of indices, we can assume that these sequences have limit points $z_\infty=\lim_n z_n$ and $w_\infty=\lim_n w_n$. It holds:
\begin{itemize}
\item $w_\infty\in W^\ss(z_\infty)$,
\item $d(\psi_t(w_\infty),\psi_t(z_\infty))\le\varepsilon$ for every $t\in\bR$, and
\item $d(w_\infty,z_\infty)\ge a$.
\end{itemize}
But this contradicts expansivity.

Let us now consider an arbitrary (not necessarily recurrent and in $\cR_1$) point $x\in X$. Given $D>0$ and $a>0$, let $T=T(2D,a)>0$ satisfying the claimed property for any quotient of a generalized rank one vector. Let $y\in W^\ss(x)$ satisfying $d(y,x)\le D$.
Choose vectors $\theta\in\chi^{-1}(x)$ and $\eta\in\chi^{-1}(y)$ and consider the minimal connected subset of $\cF^\s(\theta)$ containing $\theta$ and $\eta$, denote it by $\cF^\s(\theta,\eta)$.

By continuity of the flow $\psi_t$ on the compact space $X$, for every $t_0\ge T$ there exists $\delta_1\in(0,\frac12 D)$ such that
\[
	d(\psi_t(z),\psi_t(y))\le a
	\quad\text{ for every }
	z,y\in X, d(z,y)\le\delta_1,
	\text{ for every }
	t\in[0,t_0] .
\]
By uniform continuity of the quotient $\chi\colon T^1M\to X$, there is $\delta_2>0$ such that every set of diameter at most $\delta_2$ is quotient into a set of diameter at most $\delta_1$. Recall that, by Theorem \ref{thmy:manifolds} (3), $\cR_1$ is dense in $T^1M$. Also recall that the foliation $\cF^\s$ is continuous. Hence, there exist $\theta'\in\cR_1$, $d_\Sak(\theta',\theta)<\delta_2$, and $\eta'\in\cF^\s(\theta')$ such that $d_\Sak(\eta',\eta)\le\delta_2$ and that $\cF^\s(\theta',\eta')$ is contained in a $\delta_2$-neighborhood of $\cF^\s(\theta,\eta)$.
Letting $x'\eqdef\chi(\theta')$ and $y'\eqdef\chi(\eta')$, hence $y'\in W^\ss(x')$ and $d(y',y)\le\delta_1$ and $d(x',x)\le\delta_1$. In particular, $d(x',y')\le2\frac D2+D=2D$. Hence, by our choice of $T$, it holds
\[
	d(\psi_t(y'),\psi_t(x'))
	\le a
	\quad\text{ for every }t\in[T,t_0]
\]
Hence, by the triangle inequality
\[
	d(\psi_t(y),\psi_t(x))
	\le d(\psi_t(y),\psi_t(y'))+d(\psi_t(y'),\psi_t(x'))+d(\psi_t(x'),\psi_t(x))
	\le 3a
\]
for every $t\in[T,t_0]$. 
As $t_0\ge T$ was arbitrary, this concludes the proof of the proposition.
\end{proof}

\subsection{Local product structure}\label{sec:locpro}

Let us investigate the local product structure of the quotient flow $\Psi$. First recall some definitions. Given $x\in X$, define the \emph{center stable set} of $x$ (with respect to the quotient flow $\Psi$) by
\[
	\cW^\cs(\Psi,x)
	\eqdef \{y\in X\colon d(\psi_t(y),\psi_t(x))\le C
	\text{ for all }t\ge0\text{ for some }C>0\}
\]
 and for $\varepsilon>0$ let
\[
	\cW^\cs_\varepsilon(\Psi,x)
	\eqdef \{y\in \cW^\cs(\Psi,x)\colon d(\psi_t(y),\psi_t(x))\le \varepsilon
	\text{ for all }t\ge0\}.
\]
Analogously, define the \emph{center unstable set} $\cW^\cu(\Psi,x)$ of $x$ (with respect to the quotient flow $\Psi$) as the center stable set of $x$ (with respect to the quotient flow $\Psi^{-1}$ defined by $\Psi^{-1}(t,\cdot)=\Psi(-t,\cdot)$).  

The flow $\Psi$ has a \emph{local product structure} if for every $\varepsilon>0$ there is $\delta>0$ such that for every $x,y\in X$ with $d(x,y)\le\delta$ there is a unique $\tau=\tau(x,y)\in\bR$ with $\lvert\tau\rvert\le\varepsilon$ satisfying $\cW^\cs_\varepsilon(\Psi,\psi_\tau(x))\cap\cW^\cu_\varepsilon(\Psi,y) \ne\emptyset$.
	
\begin{proposition}\label{prop:LPS}
	The quotient flow $\Psi$ has a local product structure.	
\end{proposition}

\begin{proof}
	By Lemma \ref{lem:recregio}, the collection \eqref{eq:bases} provides a basis for the quotient topology of $\bar X$.  As $X$ is compact and locally homeomorphic to $\bar X$, there exists a finite collection 
\[
	\bar\cA
	\eqdef\{\bar A_k\}_k,
	\quad
	\bar A_k=\big\{\bar\chi(A_{\bar\theta_k}
		(\tau_k,\varepsilon_k,\delta_k,\bar\xi^-_k,\bar\xi^+_k,
			\bar\eta^-_k,\bar\eta^+_k))\big\}
\] 
of open sets such that $\cA\eqdef\{A_k\}_k$, where $A_k\eqdef\bar\Pi(\bar A_k)$ for every $k$, is an open cover of $X$. If $\kappa>0$ is a Lebesgue number for this cover then every pair of points $x,y\in X$ satisfying $d(y,x)\le\kappa$ is simultaneously contained in some element $A\in\cA$. Hence, using the notation in Section \ref{secteo:quotient}, it holds
\[
	x,y\in
	\bar\chi(\phi_{(-\tau,\tau)}(U))
\]  
for some $\tau>0$ and $U\subset\Sigma=\Sigma_{\bar\theta}(\varepsilon,\delta)$ for some positive numbers $\varepsilon$ and $\delta$ and $\bar\theta\in T_1\tilde M$. In particular, there are vectors $\bar\xi,\bar\eta\in \bar\chi^{-1}(A)$ such that $(\bar\Pi\circ\bar\chi)(\bar\xi)=x$ and $(\bar\Pi\circ\bar\chi)(\bar\eta)=y$, times $r,s\in(-\tau,\tau)$ such that $\phi_r(\bar\xi),\phi_s(\bar\eta)\in\Sigma_{\bar\theta}(\varepsilon,\delta)$ and that
\[
	[\phi_r(\bar\xi),\phi_s(\bar\eta)]
	= W^\s_\Sigma(\phi_r(\bar\xi))\cap W^\u_\Sigma(\phi_s(\bar\eta))
	= \Pi_\Sigma\big(\tilde\cF^\s(\phi_r(\bar\xi))\big)\cap 
		\tilde\cF^\cu(\phi_s(\bar\eta))
	\subset\Sigma.
\]
Hence, applying the quotient map, we can define
\[
	[x,y]
	\eqdef(\bar\Pi\circ\bar\chi)
		(\tilde\cF^\s(\bar\xi)\cap \tilde\cF^\u(\phi_{r+s}(\bar\eta)))
	\subset W^\ss(x)\cap W^\uu(\psi_{r+s}(y)).
\]	
Notice that it follows from expansivity that $[x,y]$ indeed contains just one point.
Moreover, there is some positive number $\varepsilon'$ (not depending on $A\in\cA$) such that in fact
\[
	[x,y]\in W^\ss_{\varepsilon'}(x)\cap W^\uu_{\varepsilon'}(\psi_{r+s}(y)).
\]
By Corollary \ref{cor:lem3ooo}, there exists $\delta'>0$ such that for all $t\ge0$ it holds
\[
	d\big(\psi_t([x,y]),\psi_t(x)\big)
	\le \delta'
\]
and analogously 
\[
	d\big(\psi_{-t}([x,y]),\psi_{-t+r+s}(y)\big)
	\le\delta'.
\]
This concludes the proof of the local product structure of $\Psi$.
\end{proof}

\subsection{Proof of Theorem \ref{teo:quotientflow}}\label{sec:poooof}

Together with Propositions \ref{pro:expansive} and \ref{prop:LPS}, it only remains to show that $\Psi$ is topologically mixing, which is an immediate consequence of semi-conjugacy \eqref{eq:semiconju} and the fact that $\Phi$ is mixing (recall Theorem \ref{teo:2drei}). 
\qed

\section{Lyapunov exponents and the Riccati equation}\label{sec:LyapRicc}

One of the landmarks of the theory of manifolds without conjugate points is the work of Eberlein  \cite{kn:Eberlein} linking linear independence of Green subspaces
with hyperbolicity: Green subspaces are linearly independent at every $\theta \in T_{1}{M}$ if and only
if the geodesic flow is Anosov (\cite[Theorem 3.2]{kn:Eberlein1}). Knieper \cite[Chapter IV]{kn:Knieper} for a compact surface without conjugate points with genus  greater than one shows that if the Green subspaces  vary continuously then the metric entropy of the geodesic flow with respect to the Liouville measure is positive.
There are two main features of the dynamics of the geodesic flow that are crucial in Knieper's result: Katok's proof of the existence of a hyperbolic invariant measure
for the geodesic flow \cite{Kat:82} and the Ma\~{n}\'{e}--Freire
 formula for the metric entropy of the geodesic flow in a compact manifold without conjugate points \cite{kn:FM}.
This formula is written in terms of the Riccati equation associated to the Jacobi equation, we explain briefly the main properties of this equation in the
next subsection.

\subsection{Riccati equation}

Given a geodesic $\gamma$, let $E\colon \bR\to T^1M$ be (one of the two) continuous orthogonal unit vector fields along $\gamma$. Then any orthogonal Jacobi field along $\gamma$ is given by $J(t)=j(t)E(\gamma(t))$, where $j$ is a scalar function which must satisfy the scalar differential equation
\begin{equation}\label{eq:Jacobi-2}
	\frac{d^2}{dt^2}j(t)+K(\gamma(t))j(t)=0,
\end{equation}
where $K$ is the Gaussian curvature. Assuming $j\ne0$, its logarithmic derivative  $u\eqdef \frac1j\frac{d}{dt}j$ satisfies the Riccati equation
\begin{equation}\label{eq:Riccati}
	\frac{d}{dt}u(t)+u(t)^2+K(\gamma(t))=0.
\end{equation}

On the other hand, any global solution $u\colon\bR\to\bR$ of  \eqref{eq:Riccati} by
\begin{equation}\label{def:ju}
	j(t)
	= e^{\displaystyle\int_0^t u(s)\,ds},\quad\text{ hence }
	\frac{d}{dt}j(t)
	= j(t)u(t)
	= u(t) \,e^{\displaystyle\int_0^t u(s)\,ds},
\end{equation}
defines a (normalized to $\lVert J(t)\rVert=1$) solution for the scalar equation \eqref{eq:Jacobi-2}. More precisely, denote by $u^\s_r(\theta,t)$ and $u^\u_r(\theta,t)$ the solutions of  the Riccati equation \eqref{eq:Riccati} that satisfy $u^\s_r(,\theta,-r)=\infty$ and $u^\u_r(\theta,r)=-\infty$, respectively. Then those solutions are defined for all $t>-r$ and all $t<r$, respectively. Their limit solutions
\[
	u^\ast_\theta(t)
	\eqdef \lim_{r\to\infty}u^\ast_r(\theta,t),
	\quad\ast\in\{\s,\u\},
\]
are defined for all $t\in\bR$. It holds $u^\u_\theta-u^\s_\theta\ge0$. Moreover, any global solution $u(t)$ of \eqref{eq:Riccati} is bounded and tends to $u^\u_\theta(t)$ as $t\to\infty$ and to $u^\s_\theta(t)$ as  $t\to-\infty$. The functions $u^\s_\theta(t)$ and $u^\u_\theta(t)$) are upper semi-continuous and lower semi-continuous in $\theta$, respectively. Clearly, both solutions are invariant in the sense that 
\[
	u^\ast_{\phi_s(\theta)}(t)=u^\ast_\theta(t+s),
	\quad \ast\in\{\s,\u\}.
\]	 
By \eqref{def:ju}, they define the stable and the unstable Green Jacobi fields, respectively.

The following result summarizes properties of the solutions of the Riccati equation  (see \cite[Lemma 2.8]{kn:Eberlein1}) which we will use below.

\begin{lemma} \label{lemprop:awayfromvertical}
	Let $(M,g)$ be a compact manifold. Given a geodesic $\gamma\colon \mathbb{R} \to M$, let $\kappa>0$ be a constant such that $K > -\kappa^2$. Then any solution $u(t)$ of the Riccati equation \eqref{eq:Riccati} that is defined for every $t\in (a,b)$ satisfies
\[
	  -\kappa\coth(\kappa(b-t)) 
	  \leq  u(t) 
	  \leq \kappa\coth(\kappa(t-a)). 
\]	  
In particular, the following holds:
\begin{enumerate}
\item For any  $\epsilon\in(0, b-a)$, there exists $C(\epsilon, k_{0})>0$ such that for every $t > a+\epsilon$
\[
	\lvert u(t) \rvert \leq C (\epsilon,\kappa).
\]
\item  If $u(t)$ is defined for every $t \in {\mathbb R}$ then it holds
\[
	\lvert u(t)\rvert
	\le\kappa.
\]	
\end{enumerate}
\end{lemma}

\begin{remark}[canonical construction of Green subbundles]\label{rem:Ricc}
The un-/stable Green Jacobi fields (or, in the case of surfaces, equivalently as explained above, the globally defined un-/stable solutions of the Riccati equation) completely encode the stable (unstable) Green subbundles (recall \cite[Proposition 1.7]{kn:Eberlein1}). 

The assumption that Green subbundles vary continuously is equivalent to continuity (in $\theta$) of the solutions $(J^\s_\theta(t),{J^\s_\theta}'(t))$ and $(J^\u_\theta(t),{J^\u_\theta}'(t))$ of the Jacobi equation, which in turn is equivalent to the continuous (in $\theta$) dependence of the stable and unstable solutions of the Riccati equation $u^\s_\theta(t)$ and $u^\u_\theta(t)$.

Finally, by uniqueness of solutions of the equation \eqref{eq:Riccati}, if $u^\u_\theta(t)= u^\s_\theta(t)$ for some $t$ then $u^\u_\theta\equiv u^\s_\theta$. In particular, if there exist two distinct solutions  of  \eqref{eq:Riccati} that are defined for all $t\in\bR$ then they define two linearly independent Green subbundles along $\gamma_\theta$  and hence $\theta\in\cR_1$. 
\end{remark}

\subsection{Lyapunov exponents and Green subspaces}\label{sec:greenspas}
The \emph{Lyapunov exponent} for $\theta\in T^1M$ and $v\in T_\theta T^1M$ (with respect to the geodesic flow $\Phi$) is defined by
	\[
	\lambda(\theta,v)
	\eqdef \lim_{t\to\pm\infty}\frac1t\log\,\lVert D\phi_t(v)\rVert,
	\]
	provided both limits exist and coincide. In general, limits may not exist; and if they do exist they may not coincide. By Oseledets' theorem there is a subset $\Lambda\subset T^1M$ of total probability%
\footnote{A measurable subset $\Lambda$ is of \emph{total probability} if it has full measure with respect to any invariant Borel probability measure.}
such that for every $\theta\in\Lambda$ there exist $k(\theta)\le 2n-1$ and a decomposition $T_\theta T^1M=E^1(\theta)\oplus\ldots\oplus E^{k(\theta)}(\theta)$ into invariant subspaces and numbers $\lambda_1(\theta)<\ldots<\lambda_{k(\theta)}(\theta)$ such that $\lambda(\theta,\xi)=\lambda_i(\theta)$ for every $\xi\in E^i(\theta)\setminus\{0\}$. Denote  $E^\s(\theta)\eqdef \Span\{E^i(\theta)\colon\lambda_i(\theta)<0\}$, $E^\u(\theta)\eqdef \Span\{E^i(\theta)\colon\lambda_i(\theta)>0\}$, and let $E^\c(\theta)\eqdef \Span\{E^i(\theta)\colon\lambda_i(\theta)=0\}$. Note that the latter contains $\dot\gamma_\theta(0)$. Note that 
\begin{equation}\label{eq:FreMan}
	E^\star(\theta)\subset G^\star(\theta)\subset E^\star(\theta)\oplus E^\c(\theta), 
	\quad\star=\s,\u,
\end{equation}	 
in a set of total probability (see, for example, \cite{kn:FM}).
We call the set $\Lambda$ the  \emph{set of Oseledets regular points}. 

We call an $\phi_1$-ergodic Borel probability measure $\mu$ \emph{hyperbolic} if at $\mu$-almost every point the only  subspace in the Oseledets decomposition that is associated to a zero Lyapunov exponent is the one generated by the vector field of the flow.

The relationship between nonzero Lyapunov exponents and the linear independence of Green subspaces goes back to Eberlein's characterization of Anosov geodesic flows in \cite{kn:Eberlein2}, later Freire--Ma\~{n}\'{e}'s work \cite{kn:FM} made an important contribution that was subsequently explored  by Knieper \cite{kn:Knieper}.
By \eqref{eq:FreMan}, Oseledets subbundles are naturally related to the  Green bundles.
It is natural to ask whether, as a sort of converse of Theorem \ref{thmy:manifolds}, the existence of positive Lyapunov exponents implies the linear independence of Green subspaces. 
Arnaud \cite{kn:Arnaud} answers this type of question positively  in the context of Mather measures of Tonelli Hamiltonians.
We would like to extend this result to our context, starting by the following result. 
Note that by \eqref{eq:Lya}, hypothesis \eqref{eq:hypJa} is equivalent to assuming the existence of a positive (forward) Lyapunov exponents.

\begin{proposition} \label{pro:exponents2}
	Let $(M,g)$ be a compact surface without conjugate points. Suppose that there are a geodesic $\gamma_{\theta}$ and a orthogonal Jacobi field $J(t)$ of $\gamma_\theta$ that does not vanish for every $t\geq 0$ such that
\begin{equation}\label{eq:hypJa}
	\lim_{t \to  \infty} \frac{1}{t}\log  \,\lVert J(t)\rVert
	= \lambda >0.
\end{equation}
	Then
	\begin{enumerate}
		\item There exists a orthogonal Jacobi field $W(t)$ in $\gamma_{\theta}$ such that
		$$
		\lim_{t \to  \infty} \frac{1}{t}\log\,  \lVert W(t)\rVert = -\lambda .
		$$
		\item The Jacobi field $W(t)$ is a stable Green Jacobi field.
		\item Moreover, assuming also that $(M,g)$ has continuous  stable and unstable Green bundles, then these Green subspaces are linearly independent along the orbit of $\theta$, that is, $\theta\in\cR_1$.
	\end{enumerate}
\end{proposition}

\begin{proof}
	Assuming that (1) holds true, Item (2) is straightforward from Lemma \ref{Green-surfaces}. To see that Item (1) holds true, fix $E\colon \bR\to T^1M$  an orthogonal continuous unit vector field along $\gamma_\theta$. Write the Jacobi field as $J(t)=j(t)E(t)$. By hypothesis, $j(t)\ne0$ for all $t\ge0$. Observe that  the function
\[ 
	z(t) 
	\eqdef j(t) \int_{0}^{t} \frac{1}{j^{2}(s)}\,ds 
\]	
is well defined for $t\ge0$ and is a solution of \eqref{eq:Jacobi-2} (apply a variation of parameters-argument). By hypothesis \eqref{eq:hypJa}, it holds
\[
	\lim_{t\to\infty}\frac1t\log j(t)>0
\]
and hence the following limit exists
\[
	\lim_{t \to  \infty}\int_{0}^{t} \frac{1}{j^{2}(s)}\,ds 
	= \int_{0}^{\infty} \frac{1}{j^{2}(s)}\,ds \eqdef L
\]	
and we can write
\[
	 z(t) 
	 = j(t) \int_{0}^{t} \frac{1}{j^2(s)}\,ds 
	 = j(t) ( L - \int_t^{\infty} \frac{1}{j^2(s)}\,ds ) 
	 = Lj(t) - j(t) \int_t^{\infty} \frac{1}{j^2(s)}\,ds .
\]	 
It follows that $w\colon[0,\infty)\to\bR$ defined by
\[
	w(t) 
	\eqdef j(t) \int_t^{\infty} \frac{1}{j^{2}(s)}\,ds
\]	 
also satisfies \eqref{eq:Jacobi-2}. It follows from \eqref{eq:hypJa} that for every $\epsilon >0$ there exists $T>0$ such that for every $ t >T$ it holds
\[ 	
	e^{(\lambda - \epsilon)t}  \leq j(t) \leq e^{(\lambda + \epsilon)t}.
\]	
 Let us take $\epsilon < \lambda/4$. This implies that for every $t >T$
\[
 	w(t)
	= j(t) \int_t^{\infty} \frac{1}{j^{2}(s)}\,ds 
	\leq \frac{e^{(\lambda + \epsilon)t}}{(\lambda - \epsilon)e^{2(\lambda - \epsilon)t}} 
	= \frac{e^{(-\lambda + 3\epsilon)t}}{\lambda - \epsilon}.
\]	
 Therefore,
	$$ \lim_{t \to  \infty} \frac{1}{t}\log w(t)  \leq -\lambda + 3 \epsilon , $$
	and since $\epsilon$ can be chosen arbitrarily small we conclude that
	$$\lim_{t \to  \infty} \frac{1}{t}\log w(t)  \leq -\lambda .$$
	A lower bound for this limit
	can be obtained similarly, to get
	$$\lim_{t \to  \infty} \frac{1}{t}\log w(t)  = -\lambda.$$ Taking $W(t)\eqdef w(t)E(t)$ implies Item (1).
	
To show Item (3), assume now that stable and unstable Green subspaces vary continuously. As before, write the given Jacobi field as $J(t)=j(t)E(t)$. In terms of the logarithmic derivative $u\colon[0,\infty)\to\bR$ of $j$ (also using that $j\ne0$) it holds
\[
	\lim_{t \to  \infty} \frac{1}{t}\log j(t) 
	= \lim_{t \to  \infty} \frac{1}{t}\int_{0}^{t}u(s) \,ds,
	\quad\text{ where }\quad
	u(t) = \frac{1}{j(t)}\frac{d}{dt}j(t).
\]	
Considering analogously the logarithmic derivative of $w$,  
\[
	u^\s(t)
	\eqdef \frac{1}{w(t)}\frac{d}{dt}w(t),
\]
it follows
\[
 	2\lambda 
	= \lim_{t \to  \infty} \frac{1}{t}(\log  j(t) - \log w(t) ) 
	= \lim_{t \to  \infty} \frac{1}{t}\int_{0}^{t}(u(s) - u^\s (s))\,ds.
\]	
Hence, there exists a sequence $t_{n} \to  \infty$ such that $u(t_{n}) - u^\s (t_{n}) \geq 2\lambda$ for every $n$. Note that $u(t)$ and $u^\s(t)$ both are solutions of the Riccati equation \eqref{eq:Riccati} for all $t\ge0$. 

If $u$ and $u^\s$ would be defined already for all $t\in\bR$ then by Remark \ref{rem:Ricc} the claim would follow immediately.  As this is not the case, we need to following arguments.
	Consider the geodesics $\beta_{n}(t) \eqdef \gamma_{\theta}(t+ t_{n})$ and the solutions of the Riccati equations of $\beta_{n}$ given by $u_n(t)\eqdef u(t+t_{n})$ and $u^\s_n(t)\eqdef u^\s(t+t_n)$.
	Let $(\phi_{t_{n_{k}}}(\theta))_k$  be a convergent subsequence and denote  its limit by $\eta$.
	
	\begin{claim}
		The Green subspaces $G^\s(\eta)$ and $G^\u(\eta)$ are linearly independent.
	\end{claim}
	
\begin{proof}	
Since by hypothesis Green subspaces vary continuously, the stable solutions $u^\s _{n_{k}}(t) = u^\s (t+ t_{n_{k}})$ of the Riccati equation for $t\mapsto \beta_{n_k}(t)$ converge to the stable solution $u^\s _{\eta}\colon\bR\to \bR$ for $\gamma_\eta$ (recall Remark \ref{rem:Ricc}). 

The sequence of solutions $u_{n_{k}}(t)$  for $ \beta_{n_k}$ has a subsequence converging to some solution of the Riccati equation $\bar{u}(t)$ defined for every $ t \in \mathbb{R}$ by Lemma \ref{lemprop:awayfromvertical}. Indeed, $t\mapsto u_{n_{k}}(t)$ are uniformly bounded for every $t \geq -t_{n_{k}} + 1$ and equicontinuous in this interval since their derivatives are uniformly bounded by the Riccati relation $\ddot u(t) = -u^{2} - K$. Since $u(t_{n}) - u^\s (t_{n}) \geq 2\lambda$ for every $n>0$, the same inequality holds true in the limit, that is, $\bar{u}(0) - u^\s _{\eta}(0) \geq 2\lambda$. 

Hence the geodesic $\gamma_{\eta}$ has two different solutions of the Riccati equation that are defined for every $t \in \mathbb{R}$: the stable solution $u^\s _{\eta}(t)$ and $\bar{u}(t)$.
	We have that $\bar{u}(t) > u^\s _{\eta}(t)$ for every $t\in \mathbb{R}$ by uniqueness of solutions of ordinary differential equations. Therefore, the unstable solution $u^\u _{\eta}(t)$, that is the supremum of the solutions defined for every $ t \in \mathbb{R}$, is strictly greater than $u^\s _{\eta}(t)$. This together with Remark \ref{rem:Ricc} yields the Claim.
	\end{proof}
	
	Finally, notice that $\eta$ is a limit point of the orbit of $\theta$, and Green subspaces at $\eta$ are linearly independent. By continuity of Green bundles,
	there exists an open set which contains $\eta$ where Green subspaces are linearly independent, so the orbit of $\theta$ meets this open set. By invariance
	of Green subspaces, the Green subspaces are linearly independent along the entire orbit of $\theta$. This finishes the proof of Item (3).
\end{proof}

\begin{remark} 
Observe that Item (3) in Proposition \ref{pro:exponents2} is false without assuming the continuity of Green bundles. Indeed, \cite{kn:BBB} provides an example of a compact surface without conjugate points where Green bundles are not continuous and which exhibits a geodesic $\gamma_{\theta}$ where $G^\s(\theta) = G^\u(\theta)$ and the Lyapunov exponent in this (unique) Green subspace is positive. 
\end{remark}

\section{Entropy}\label{sec:entropy}

In this section we assume that $(M,g)$ is a compact surface without conjugate points of genus  greater than one with continuous stable and unstable Green bundles. 

The goal of this section is to show that the entropy of the geodesic flow in any nontrivial strip vanishes and to prove Theorem \ref{main2}. We examine both metric and topological entropies.

A Borel probability measure on a metric space is \emph{invariant} under a continuous flow $\Psi=(\psi_t)_{t\in\bR}$ on $X$ if it is $\psi_t$-invariant for every $t\in\bR$. We say that $Z\subset X$ is \emph{invariant} under the flow if $\psi_t(Z)=Z$ for every $t\in\bR$. An invariant measure is \emph{ergodic} if every invariant set has either measure one or measure zero.

Recall that the \emph{topological entropy} of a compact set $Z\subset T^1M$ (with respect to the time-1 map $\phi_1$) is defined by
\[
	h(\phi_1,Z)
	\eqdef \lim_{\varepsilon\to0}\limsup_{n\to\infty}\frac1n\log M(n,\varepsilon,Z),
\]
where $M(n,\varepsilon,Z)$ denotes the maximal cardinality of any $(n,\varepsilon)$-separated subset $E\subset Z$. A set $E$ is  \emph{$(n,\varepsilon)$-separated} if $x,y\in E$, $x\ne y$, implies $d(\phi_k(x),\phi_k(y))\ge\varepsilon$ for some $k\in\{0,\ldots,n-1\}$. 

For $Z$ compact $\phi_1$-invariant, by the variation principle \cite[Theorem 9.10]{Wal:82} 
\begin{equation}\label{eq:VP}
	h(\phi_1,Z)
	=\sup_\mu h_\mu(\phi_1),
\end{equation}
where the supremum is taken over all $\phi_1$-invariant Borel probability measures $\mu$ supported on $Z$ and where $h_\mu(\phi_1)$ denotes the \emph{metric entropy} of $\mu$ (with respect to the time-1 map $\phi_1$). The \emph{topological entropy} of $Z$ (with respect to the flow $\Phi$)  is analogously defined and denoted  by $h(\Phi,Z)$ (see \cite[Section 3]{BowRue:75}); it satisfies
\[
	h(\Phi,Z)
	= h(\phi_1,Z).
\]
A measure $\mu$ is a \emph{measure of maximal entropy} (with respect to $\Phi$) if its entropy realizes the supremum in \eqref{eq:VP}.
By Ruelle's inequality, it holds
\begin{equation}\label{eq:Ruelle}
	h_{\mu}(\phi_1) 
	\leq \int_{T^1M}\lambda^{+}(\theta)\,d\mu(\theta) ,
\end{equation}
where $\lambda^{+}(\theta)$ is the nonnegative Lyapunov exponent of $\theta$. 

\subsection{The entropy on strips}

For the following compare also \cite[Section 4]{LiuWan:16}.

\begin{lemma} \label{Zero-top-entropy}
	For every $\theta \in T^1M$ it holds $h(\phi_1,{\mathscr F}^\s (\theta) \cap {\mathscr F}^\u (\theta))=0$. In particular, $h(\phi_1,\chi^{-1}(\chi(\theta)))=0$.
\end{lemma}

\begin{proof}
Note that the result is trivial if ${\mathscr F}^\s (\theta) \cap {\mathscr F}^\u (\theta)=\{\theta\}$.

Let us consider now the general case. Let $\bar\theta$ be any lift of $\theta$.
By Lemma \ref{strips} there exists $Q=Q(M)>0$ such that the width of the strip $ S(\bar \theta)$ is at most $Q$. In particular, the width of $\cI(\theta)\eqdef {\mathscr F}^\s (\theta) \cap {\mathscr F}^\u (\theta)$ is at most $Q$.

Given $\varepsilon>0$ and $n\ge1$, let $E\subset \cI(\theta)$ be an $(n,\varepsilon)$-separated set. For $k\in\{0,\ldots,n-1\}$ denote by $E_k\subset E$ the set of points such that 
$x,y\in E_k$ implies
\begin{equation}\label{eq:smalldist}
 	d_g(\phi_t(x), \phi_t(y)) \geq \epsilon 
	\quad
	\text{ for some }\quad
	t\in[k,k+1).
\end{equation}
Then $E=\bigcup_{k=0}^{n-1}E_k$. 
Let us estimate the cardinality of $E_k$.
By compactness of $(M,g)$, there exists $\delta_1>0$ such that \eqref{eq:smalldist} implies
\[
 	d_g(\phi_t(x), \phi_t(y)) \geq \delta_1 
	\quad
	\text{ for every }\quad
	t\in[k,k+1).	
\]
Denote by $d^\s_g(x_1,x_2)$ the intrinsic distance of two points $x_1,x_2\in{\mathscr F}^\s(\eta)$. To be more precise, consider a curve $\zeta\colon[0,1]\to{\mathscr F}^\s(\eta)$ with $\zeta(0)=x_1$ and $\zeta(1)=x_2$ and let $d^\s_g(x_1,x_2)$ be the length of its canonical projection to $M$. Now recall that the sets ${\mathscr F}^\s (\theta), {\mathscr F}^\u (\theta) $ are smooth with $L$-Lipschitz first derivatives where $L>0$ is uniform in $T^1M$ (Remark \ref{rem:smoothmfds} and Theorem \ref{thmy:manifolds} (3)). Hence, it follows that there exists $\delta_2>0$ such that \eqref{eq:smalldist} implies
\[
 	d^\s_g(\phi_t(x), \phi_t(y)) \geq \delta_2
	\quad
	\text{ for every }\quad
	t\in[k,k+1).	
\]
This together implies that
\[
	\delta_2\card E_k \le Q.
\]

Thus,
\[
	\card E
	\le \sum_{k=0}^{n-1}\card E_k
	\le \sum_{k=0}^{n-1}\delta_2^{-1}Q
	= n\delta_2^{-1}Q.
\]
This immediately implies 
\[
	h(\phi_1,\cI(\theta))
	\le \lim_{\varepsilon\to0}\limsup_{n\to\infty}\frac1n\log(n\delta_2^{-1}Q)
	=0,
\]
proving the lemma.
\end{proof}

Lemma \ref{Zero-top-entropy} and \cite[Theorem 17]{Bow:71} together imply the following result.

\begin{lemma}
	For every compact invariant set $Z\subset T^1M$ it holds $h(\phi_1,Z)=h(\psi_1,\chi(Z))$.
\end{lemma}

Note that $h$-expansiveness stated in the proof of the following result was shown in \cite{LiuWan:16}, for completeness we provide an independent proof.

\begin{proposition} \label{pro:zero-entropy}
The metric entropy (with respect to $\phi_1$) of any invariant measure supported in the set $T^1M\setminus \cR_1$ is zero and the topological entropy of $T^1M\setminus\cR_1$ (with respect to $\phi_1$) is zero.

Moreover, the entropy map $\mu\mapsto h_\mu(\phi_1)$ is upper semi-continuous.
\end{proposition}

\begin{proof}
Let $\mu $ be an invariant measure supported in $T^1M\setminus \cR_1$. 
Since the set of Oseledets regular points of $\mu$ has probability one, it suffices to
evaluate the above integral on the set of Lyapunov regular points, only. Together with \eqref{eq:Ruelle}, it follows immediately from Proposition \ref{pro:exponents2} Item (3), that $\mu$-almost every $\theta$ satisfies $\lambda^+(\theta)=0$. This proves the first claim. 

The second claim is now an immediate consequence of \eqref{eq:VP} applied to the closed invariant set $T^1M\setminus \cR_1$.

By \eqref{eq:semiconju}, the time-1 map $\psi_1\colon X\to X$ is a (topological) factor of the time-1 map $\phi_1\colon T^1M\to T^1M$. To show upper semi-continuity of the entropy map, first recall that by \cite{LedWal:77}, 
\[
	\sup_{\mu\colon\chi_\ast\mu=\nu}h_\mu(\phi_1)
	= h_\nu(\psi_1)+\int h(\phi_1,\chi^{-1}(x))\,d\nu(x).
\]
It follows from Lemma \ref{Zero-top-entropy} and the definition of the factor map $\chi$ that the latter integral is zero. Hence,  for every $\mu\in\cM(\phi_1)$ and $\nu=\chi_\ast\mu$  it holds
\[
	h_\nu(\psi_1)
	= h_\mu(\phi_1).
\]  
Let $(\mu_n)_n\subset\cM(\phi_1)$ be a sequence weak$\ast$ converging to some measure $\mu$. Then by continuity of the factor map and hence of the push forward $\chi_\ast$ it follows that $\nu_n\eqdef\chi_\ast\mu_n$ weak$\ast$ converges to $\nu\eqdef\chi_\ast\mu$.
By Proposition \ref{pro:expansive}, the quotient flow $\Psi$ is expansive. Hence, its time-1 map  $\psi_1$ is $h$-expansive, that is, there exists $\varepsilon>0$ so that for every $x\in X$ the set
\[
	\{y\in X\colon d(\psi_n(y),\psi_n(x))\le\varepsilon\text{ for all }n\in\bZ\}
\]
has zero topological entropy (with respect to $\psi_1$, compare for example \cite[Example 1.6]{Bow:72b}). The latter implies that 
its entropy map is upper semi-continuous and hence $h_\nu(\psi_1) \ge\limsup_nh_{\nu_n}(\psi_1)$. This implies $h_\mu(\phi_1) \ge\limsup_nh_{\mu_n}(\phi_1)$.
\end{proof}

Proposition \ref{pro:zero-entropy} together with \eqref{eq:VP} guarantee the existence of an ergodic measure of maximal entropy $h_\mu(\phi_1)=h(\phi_1,T^1M)$. It remains to show that it is unique. 
First, it follows from Theorem \ref{teo:quotientflow} together with Franco \cite{Fra:77} that there is a unique (hence ergodic) measure of maximal entropy with respect to the quotient flow $\Psi$ (see also \cite[Corollary 6.6]{GelRug:19}). 

\begin{lemma}\label{lem:last}
	The measure of maximal entropy $\nu$ (with respect to $\Psi$) satisfies
$\nu(\{\chi(\theta)\colon[\theta]=\{\theta\}\})=1$.
\end{lemma}

\begin{proof}
	By definition, $\{\chi(\theta)\colon[\theta]=\{\theta\}\}=\chi(\cR_1)$. By Theorem \ref{thmy:manifolds}, $\cR_1$ and its complement $T^1M\setminus\cR_1$ both are invariant under the geodesic flow. Hence, as $\Psi$ is a factor, it follows that $\chi(\cR_1)$ and its complement are both invariant under the quotient flow $\Psi$. 
By ergodicity, only one of these sets has full measure $\nu$. The claim now follows from Proposition \ref{pro:zero-entropy}.
\end{proof}

\begin{proof}[Proof of Theorem \ref{main2}]
The claim follows from \cite[Theorem 6.7]{GelRug:19} (which is \cite[Theorem 1.5]{BuzFisSamVas:12} in our setting), together with Lemmas \ref{Zero-top-entropy} and \ref{lem:last}. 
\end{proof}

\bibliographystyle{plain}

\begin{thebibliography}{10}

\bibitem{Ano:69}
D.~V. Anosov.
\newblock {\em Geodesic flows on closed {R}iemann manifolds with negative
  curvature}.
\newblock Proceedings of the Steklov Institute of Mathematics, No. 90 (1967).
  Translated from the Russian by S. Feder. American Mathematical Society,
  Providence, R.I., 1969.

\bibitem{kn:Arnaud}
Marie-Claude Arnaud.
\newblock Green bundles, {L}yapunov exponents and regularity along the supports
  of the minimizing measures.
\newblock {\em Ann. Inst. H. Poincar\'{e} Anal. Non Lin\'{e}aire},
  29(6):989--1007, 2012.

\bibitem{kn:BBB}
W.~Ballmann, M.~Brin, and K.~Burns.
\newblock On surfaces with no conjugate points.
\newblock {\em J. Differential Geom.}, 25(2):249--273, 1987.

\bibitem{BarRug:07}
J.~Barbosa~Gomes and Rafael~O. Ruggiero.
\newblock Uniqueness of central foliations of geodesic flows for compact
  surfaces without conjugate points.
\newblock {\em Nonlinearity}, 20(2):497--515, 2007.

\bibitem{Bin:59}
R.~H. Bing.
\newblock An alternative proof that {$3$}-manifolds can be triangulated.
\newblock {\em Ann. of Math. (2)}, 69:37--65, 1959.

\bibitem{Bos:15}
Aur\'elien Bosch\'e.
\newblock {\em Expansive geodesic flows on compact manifolds without conjugate
  points}.
\newblock {\tt https://tel.archives-ouvertes.fr/tel-01691107/}. 2015.
\newblock Th\'ese, Institut Fourier and Dissertation, Fakult\"at f\"ur
  Mathematik der Ruhr-Universit\"at Bochum, 2015.

\bibitem{Bow:71}
Rufus Bowen.
\newblock Entropy for group endomorphisms and homogeneous spaces.
\newblock {\em Trans. Amer. Math. Soc.}, 153:401--414, 1971.

\bibitem{Bow:72b}
Rufus Bowen.
\newblock Entropy-expansive maps.
\newblock {\em Trans. Amer. Math. Soc.}, 164:323--331, 1972.

\bibitem{Bow:74}
Rufus Bowen.
\newblock Some systems with unique equilibrium states.
\newblock {\em Math. Systems Theory}, 8(3):193--202, 1974/75.

\bibitem{BowRue:75}
Rufus Bowen and David Ruelle.
\newblock The ergodic theory of {A}xiom {A} flows.
\newblock {\em Invent. Math.}, 29(3):181--202, 1975.

\bibitem{kn:BW72}
Rufus Bowen and Peter Walters.
\newblock Expansive one-parameter flows.
\newblock {\em J. Differential Equations}, 12:180--193, 1972.

\bibitem{kn:Burns}
Keith Burns.
\newblock The flat strip theorem fails for surfaces with no conjugate points.
\newblock {\em Proc. Amer. Math. Soc.}, 115(1):199--206, 1992.

\bibitem{BuzFisSamVas:12}
J.~Buzzi, T.~Fisher, M.~Sambarino, and C.~V\'{a}squez.
\newblock Maximal entropy measures for certain partially hyperbolic, derived
  from {A}nosov systems.
\newblock {\em Ergodic Theory Dynam. Systems}, 32(1):63--79, 2012.

\bibitem{CliKniWar:}
Vaughn Climenhaga, Gerhard Knieper, and Khadim War.
\newblock Uniqueness of the measure of maximal entropy for geodesic flows on
  certain manifolds without conjugate points.
\newblock arXiv:1903.09831.

\bibitem{CouSch:14}
Yves Coud\`ene and Barbara Schapira.
\newblock Generic measures for geodesic flows on nonpositively curved
  manifolds.
\newblock {\em J. \'{E}c. polytech. Math.}, 1:387--408, 2014.

\bibitem{CroFatFel:92}
C.~Croke, A.~Fathi, and J.~Feldman.
\newblock The marked length-spectrum of a surface of nonpositive curvature.
\newblock {\em Topology}, 31(4):847--855, 1992.

\bibitem{kn:Croke}
Christopher~B. Croke.
\newblock Rigidity for surfaces of nonpositive curvature.
\newblock {\em Comment. Math. Helv.}, 65(1):150--169, 1990.

\bibitem{kn:DoCarmo}
Manfredo Perdig\~{a}o do~Carmo.
\newblock {\em Riemannian geometry}.
\newblock Mathematics: Theory \& Applications. Birkh\"{a}user Boston, Inc.,
  Boston, MA, 2013.

\bibitem{kn:EO}
P.~Eberlein and B.~O'Neill.
\newblock Visibility manifolds.
\newblock {\em Pacific J. Math.}, 46:45--109, 1973.

\bibitem{kn:Eberlein}
Patrick Eberlein.
\newblock Geodesic flow in certain manifolds without conjugate points.
\newblock {\em Trans. Amer. Math. Soc.}, 167:151--170, 1972.

\bibitem{kn:Eberlein2}
Patrick Eberlein.
\newblock Geodesic flows on negatively curved manifolds. {II}.
\newblock {\em Trans. Amer. Math. Soc.}, 178:57--82, 1973.

\bibitem{kn:Eberlein1}
Patrick Eberlein.
\newblock When is a geodesic flow of {A}nosov type? {I},{II}.
\newblock {\em J. Differential Geometry}, 8:437--463; ibid. 8 (1973), 565--577,
  1973.

\bibitem{Ebe:77}
Patrick Eberlein.
\newblock Horocycle flows on certain surfaces without conjugate points.
\newblock {\em Trans. Amer. Math. Soc.}, 233:1--36, 1977.

\bibitem{kn:Eschenburg}
Jost-Hinrich Eschenburg.
\newblock Horospheres and the stable part of the geodesic flow.
\newblock {\em Math. Z.}, 153(3):237--251, 1977.

\bibitem{Fra:77}
Ernesto Franco.
\newblock Flows with unique equilibrium states.
\newblock {\em Amer. J. Math.}, 99(3):486--514, 1977.

\bibitem{kn:FM}
A.~Freire and R.~Ma\~{n}\'{e}.
\newblock On the entropy of the geodesic flow in manifolds without conjugate
  points.
\newblock {\em Invent. Math.}, 69(3):375--392, 1982.

\bibitem{Gel:19}
Katrin Gelfert.
\newblock Non-hyperbolic behavior of geodesic flows of rank 1 surfaces.
\newblock {\em Discrete Contin. Dyn. Syst.}, 39(1):521--551, 2019.

\bibitem{GelRug:19}
Katrin Gelfert and Rafael~O. Ruggiero.
\newblock Geodesic flows modelled by expansive flows.
\newblock {\em Proc. Edinb. Math. Soc. (2)}, 62(1):61--95, 2019.

\bibitem{kn:Ghys}
\'{E}tienne Ghys.
\newblock Flots d'{A}nosov sur les {$3$}-vari\'{e}t\'{e}s fibr\'{e}es en
  cercles.
\newblock {\em Ergodic Theory Dynam. Systems}, 4(1):67--80, 1984.

\bibitem{kn:Green}
Leon~W. Green.
\newblock Surfaces without conjugate points.
\newblock {\em Trans. Amer. Math. Soc.}, 76:529--546, 1954.

\bibitem{kn:Green2}
Leon~W. Green.
\newblock A theorem of {E}. {H}opf.
\newblock {\em Michigan Math. J.}, 5:31--34, 1958.

\bibitem{kn:Gromov}
M.~Gromov.
\newblock Hyperbolic groups.
\newblock In {\em Essays in group theory}, volume~8 of {\em Math. Sci. Res.
  Inst. Publ.}, pages 75--263. Springer, New York, 1987.

\bibitem{Gro:00}
Mikha\"{\i}l Gromov.
\newblock Three remarks on geodesic dynamics and fundamental group.
\newblock {\em Enseign. Math. (2)}, 46(3-4):391--402, 2000.

\bibitem{HeiimH:77}
Ernst Heintze and Hans-Christoph Im~Hof.
\newblock Geometry of horospheres.
\newblock {\em J. Differential Geometry}, 12(4):481--491 (1978), 1977.

\bibitem{Kat:82}
Anatole Katok.
\newblock Entropy and closed geodesics.
\newblock {\em Ergodic Theory Dynam. Systems}, 2(3-4):339--365 (1983), 1982.

\bibitem{KatHas:95}
Anatole Katok and Boris Hasselblatt.
\newblock {\em Introduction to the modern theory of dynamical systems},
  volume~54 of {\em Encyclopedia of Mathematics and its Applications}.
\newblock Cambridge University Press, Cambridge, 1995.
\newblock With a supplementary chapter by Katok and Leonardo Mendoza.

\bibitem{Kli:71}
W.~Klingenberg.
\newblock Geod\"{a}tischer {F}luss auf {M}annigfaltigkeiten vom hyperbolischen
  {T}yp.
\newblock {\em Invent. Math.}, 14:63--82, 1971.

\bibitem{kn:Knieper}
Gerhard Knieper.
\newblock {\em Mannigfaltigkeiten ohne konjugierte {P}unkte}, volume 168 of
  {\em Bonner Mathematische Schriften [Bonn Mathematical Publications]}.
\newblock Universit\"{a}t Bonn, Mathematisches Institut, Bonn, 1986.
\newblock Dissertation, Rheinische Friedrich-Wilhelms-Universit\"{a}t, Bonn,
  1985.

\bibitem{LedWal:77}
François Ledrappier and Peter Walters.
\newblock A relativised variational principle for continuous transformations.
\newblock {\em J. London Math. Soc. (2)}, 16(3):568--576, 1977.

\bibitem{LiuWan:16}
Fei Liu and Fang Wang.
\newblock Entropy-expansiveness of geodesic flows on closed manifolds without
  conjugate points.
\newblock {\em Acta Math. Sin. (Engl. Ser.)}, 32(4):507--520, 2016.

\bibitem{LiuWanWu:}
Fei Liu, Fang Wang, and Weisheng Wu.
\newblock On the patterson-sullivan measure for geodesic flows on rank 1
  manifolds without focal points.
\newblock arXiv:1812.04398v1.

\bibitem{Moi:77}
Edwin~E. Moise.
\newblock {\em Geometric topology in dimensions {$2$} and {$3$}}.
\newblock Springer-Verlag, New York-Heidelberg, 1977.
\newblock Graduate Texts in Mathematics, Vol. 47.

\bibitem{kn:Morse}
Harold~Marston Morse.
\newblock A fundamental class of geodesics on any closed surface of genus
  greater than one.
\newblock {\em Trans. Amer. Math. Soc.}, 26(1):25--60, 1924.

\bibitem{kn:Otal}
Jean-Pierre Otal.
\newblock Le spectre marqu\'{e} des longueurs des surfaces \`a courbure
  n\'{e}gative.
\newblock {\em Ann. of Math. (2)}, 131(1):151--162, 1990.

\bibitem{kn:Pesin}
Yakov~B. Pesin.
\newblock Geodesic flows in closed {R}iemannian manifolds without focal points.
\newblock {\em Izv. Akad. Nauk SSSR Ser. Mat.}, 41(6):1252--1288, 1447, 1977.

\bibitem{kn:RuggieroA}
Rafael~O. Ruggiero.
\newblock On the divergence of geodesic rays in manifolds without conjugate
  points, dynamics of the geodesic flow and global geometry.
\newblock Number 287, pages xx, 231--249. 2003.
\newblock Geometric methods in dynamics. II.

\bibitem{kn:Ruggieroensaios}
Rafael~O. Ruggiero.
\newblock {\em Dynamics and global geometry of manifolds without conjugate
  points}, volume~12 of {\em Ensaios Matem\'{a}ticos [Mathematical Surveys]}.
\newblock Sociedade Brasileira de Matem\'{a}tica, Rio de Janeiro, 2007.

\bibitem{kn:RosasRug}
Rafael~O. Ruggiero and Vladimir~A. Rosas~Meneses.
\newblock On the {P}esin set of expansive geodesic flows in manifolds with no
  conjugate points.
\newblock {\em Bull. Braz. Math. Soc. (N.S.)}, 34(2):263--274, 2003.

\bibitem{Wal:82}
Peter Walters.
\newblock {\em An introduction to ergodic theory}, volume~79 of {\em Graduate
  Texts in Mathematics}.
\newblock Springer-Verlag, New York-Berlin, 1982.

\end{thebibliography}

\end{document}